\documentclass[11pt]{amsart}

\usepackage{amssymb,amsmath,amsthm} 
\usepackage{stmaryrd} 
\usepackage{mathrsfs} 

\usepackage{enumitem} 
\setenumerate{label=(\alph*), font=\normalfont}

\usepackage[all]{xy} 

\usepackage{multicol}
\usepackage{bm}
\usepackage{rotating}

\usepackage{multirow} 
\usepackage{tabu}
\usepackage{xcolor}
\usepackage{url}

\DeclareFontFamily{OT1}{pzc}{}
\DeclareFontShape{OT1}{pzc}{m}{it}{<-> s * [1.10] pzcmi7t}{}
\DeclareMathAlphabet{\mathpzc}{OT1}{pzc}{m}{it}

\theoremstyle{plain}
\newtheorem{thm}{Theorem}[section]
\newtheorem{lem}[thm]{Lemma}
\newtheorem{prop}[thm]{Proposition}
\newtheorem{cor}[thm]{Corollary}
\newtheorem{conj}[thm]{Conjecture}

\theoremstyle{definition}
\newtheorem{defn}[thm]{Definition}

\newtheorem{observation}[thm]{Observation}
\theoremstyle{remark}

\newtheorem{rem}[thm]{Remark}

\numberwithin{equation}{section}
\title[Representation Dimensions of Algebraic Tori]{Maximal Representation Dimensions of Algebraic Tori of Fixed Dimension Over Arbitrary Fields}

\author[B.~Heath]{Bailey Heath}
\address[B.~Heath]{Department of Mathematics, Yale University}

\begin{document}

\maketitle


\begin{abstract}
We define the representation dimension of an algebraic torus $T$ to be the minimal positive integer $r$ such that there exists a faithful embedding $T \hookrightarrow \operatorname{GL}_r$.  
Given a positive integer $n$, there exists a maximal representation dimension of all $n$-dimensional algebraic tori over all fields.  In this paper, we use the theory of group actions on lattices to find lower bounds on this maximum for all $n$.  Further, we find the exact maximum value for irreducible tori for all $n \in \left\lbrace 1, 2, \dots, 10, 11, 13, 17, 19, 23\right\rbrace$ and conjecturally infinitely many primes $n$.
\end{abstract}

\section{Introduction}
\label{sec:intro}

Let $k$ be a field with separable closure $k_s$.  A natural invariant of a finite group $G$ is the minimal dimension of a faithful representation of $G$ over $k$, known as the \textit{representation dimension} $\operatorname{rdim}_k\left(G\right)$ of $G$ over $k$.  In \cite[Thm. 4.1]{Karpenko}, Karpenko and Merkurjev showed that $\operatorname{rdim}_k\left(G\right)$ is equal to the essential dimension $\operatorname{ed}_k\left(G\right)$ when $G$ is a finite $p$-group over a field $k$ containing a $p$-th root of unity.  Since then, much work has been done on the representation dimensions of finite $p$-groups, such as \cite{Meyer}, \cite{Cernele}, and \cite{Bardestani}.  For finite groups more broadly, representation dimension is an upper bound on essential dimension, although the two do not coincide in general.
\par
It is natural to study bounds on the complexity of groups via bounds on their representation dimensions.  In \cite{Cernele}, Cernele, Kamgarpour, and Reichstein proved upper bounds in terms of $p$ and $n$ on the representation dimensions over the complex numbers $\mathbb{C}$ of groups of order $p^n$ for prime $p$.  Similarly, Moret\'{o} showed that the representation dimension over $\mathbb{C}$ of any finite group $G$, except for certain $2$-groups, is at most $\sqrt{\left|G\right|}$ \cite{Moreto}.
\par
While the works in the previous paragraph focused on complex representations of groups, in the current work we investigate representations of algebraic tori over arbitrary fields.  An \textit{(algebraic) torus} over $k$ is an algebraic group over $k$ that, when viewed as an algebraic group over $k_s$, is isomorphic to a finite product of copies of the multiplicative group.  We say that an algebraic torus over $k$ is \textit{irreducible} if it does not properly contain a subtorus over $k$.  Every algebraic torus $T$ admits a matrix representation via an embedding $T \hookrightarrow \operatorname{GL}_n$ for some positive integer $n$, and the \textit{representation dimension} $\operatorname{rdim}\left(T\right)$ is the minimal $n$ for which there exists a faithful embedding $T \hookrightarrow \operatorname{GL}_n$.  A result analogous to that of \cite[Thm. 4.1]{Karpenko} for essential $p$-dimensions of tori over arbitrary fields was proved in \cite{Lotscher}, with minimal degrees of $p$-faithful representations of tori appearing in the result.  Moreover, Merkurjev proved an explicit formula for the representation dimension of an algebraic torus over a field whose splitting group is a $p$-group \cite[Thm. 4.3]{Merkurjev}.
\par
A consequence of the Jordan-Zassenhaus Theorem below is the following:  For any positive integer $n$, there exists a maximal representation dimension of all $n$-dimensional algebraic tori over all fields.  In this paper, we study the following quantities, where each maximum is taken over all fields:
\begin{align*} \label{rdim n def}
\operatorname{rdim}\left(n\right) & := \max\left\lbrace \operatorname{rdim}\left(T\right) \ | \ T \text{ a torus, }\dim\left(T\right) = n\right\rbrace \\
\operatorname{rdim}_{\operatorname{irr}}\left(n\right) & := \max\left\lbrace \operatorname{rdim}\left(T\right) \ | \ T \text{ an irreducible torus, }\dim\left(T\right) = n\right\rbrace.
\end{align*}
\par
We now state the results of the present work.

\begin{thm} \label{lower bound for rdim theorem 1}
In Table \ref{lower bound on rdim n table 1}, we have lower bounds on $\operatorname{rdim}\left(n\right)$ for all positive integers $n$.

\begin{table}[h]
\centering
\caption[Lower bounds on $\operatorname{rdim}\left(n\right)$ for all $n \in \mathbb{Z}^+$.]{Lower bounds on $\operatorname{rdim}\left(n\right)$ for all positive integers $n$.}
\label{lower bound on rdim n table 1}
\begin{tabular}{|c|c|} \hline $n$ & $\operatorname{rdim}\left(n\right) \geq $ \\\hline $1$ & $2$ \\\hline $2$ & $6$ \\\hline $3$ & $12$ \\\hline $4$ & $24$ \\\hline $5$ & $40$ \\\hline $6$ & $72$ \\\hline $\geq 7$ & $2^n$ \\\hline
\end{tabular}
\end{table}
\end{thm}

Note that Theorem \ref{lower bound for rdim theorem 1} only asserts the existence of \textit{some} $n$-dimensional torus over \textit{some} field with the claimed representation dimension, so the bounds are not guaranteed to hold over a given particular field.  However, we show that these bounds hold over the following important class of fields.

\begin{cor} \label{number field corollary 1}
The bounds on $\operatorname{rdim}\left(n\right)$ in Theorem \ref{lower bound for rdim theorem 1} hold over any number field.
\end{cor}

We show that Table \ref{lower bound on rdim n table 1} gives exact values of $\operatorname{rdim}_{\operatorname{irr}}\left(n\right)$ in many cases. 

\begin{thm} \label{symrankirr theorem 1}
The bounds in Table \ref{lower bound on rdim n table 1} give exact values of $\operatorname{rdim}_{\operatorname{irr}}\left(n\right)$ for the following values of $n$:
\begin{enumerate}
\item $n \in \left\lbrace 1, 2, \dots, 10, 11, 13, 17, 19, 23\right\rbrace$.
\item primes $n$ such that the multiplicative order of $2$ modulo $n$ is equal to $n - 1$ or $\frac{n-1}{2}$.
\end{enumerate}
\end{thm}

Item (b) in Theorem \ref{symrankirr theorem 1} addresses the specific cases $a = 1, 2$ of the following more general asymptotic bound; in particular, it says that we can take $N_1 = N_2 = 1$.

\begin{thm} \label{asymptotic theorem}
For all positive integers $a$, there exists a positive integer $N_a$ such that, for all primes $p \geq N_a$ such that the multiplicative order of $2$ modulo $p$ is equal to $\frac{p-1}{a}$, we have $\operatorname{rdim}_{\operatorname{irr}}\left(p\right) = 2^p$.
\end{thm}

Although a bound on $N_a$ in Theorem \ref{asymptotic theorem} can be extracted from the proof of Theorem \ref{p-1/a theorem}, this is not enough to prove that Theorem \ref{asymptotic theorem} implies that $\operatorname{rdim}_{\operatorname{irr}}\left(p\right) = 2^p$ for infinitely many primes $p$.  For example, this would follow from the existence of some positive integer $a$ such that there exist infinitely primes $p$ for which the order of $2$ modulo $p$ is equal to $\frac{p-1}{a}$, but no such $a$ is known.  In particular, Artin famously conjectured that there are infinitely many primes $p$ such that $2$ is a primitive root modulo $p$ (i.e., the $a = 1$ case) \cite[\S 1]{HeathBrown}, yet this remains an open problem.  Nonetheless, based on the theorems above and computational evidence, we pose the following conjecture.

\begin{conj} \label{rdim conjecture}
The bounds in Table \ref{lower bound on rdim n table 1} give exact values of $\operatorname{rdim}\left(n\right)$ for all positive integers $n$.
\end{conj}

\textbf{Acknowledgments.  }It is my pleasure to thank Dr. Alexander Duncan for suggesting this project and for his contributions and support throughout.  I am also thankful to Dr. Zinovy Reichstein for helpful comments.  This work was partially supported by a SPARC Graduate Research Grant from the Office of the Vice President for Research at the University of South Carolina.

\section{Preliminaries}

The main tools that we use to study representation dimensions of algebraic tori are $G$-lattices and their symmetric ranks.  This largely reduces to the study of actions of finite integral matrix groups on certain subgroups of the additive group $\mathbb{Z}^n$.
\par
Let us begin with the definition of a $G$-lattice.  We write $\mathbb{Z}^+$ for the set of positive integers.

\begin{defn} \label{G-lattice def}
A \textbf{lattice} $L$ is a free $\mathbb{Z}$-module of finite rank $n \in \mathbb{Z}^+$.  In this case, we say that the \textbf{rank} of $L$ is $\operatorname{rank}\left(L\right) = n$.  If the group $G$ acts on $L$ (or, equivalently, if $L$ is a module over the group ring $\mathbb{Z}\left[G\right]$), then we further say that $L$ is a \textbf{$G$-lattice}.  
\end{defn}

Note that every $G$-lattice $L$ can be written (up to isomorphism) with $L = \mathbb{Z}^n$ and $G$ a subgroup of $\operatorname{GL}_n\left(\mathbb{Z}\right)$, where $n = \operatorname{rank}\left(L\right)$.  We will often think about $G$-lattices through this linear algebraic lens moving forward.  
\par
We often focus on \textit{irreducible} $G$-lattices in this paper.

\begin{defn} \label{irreducible lattice def}
The $G$-lattice $L$ is \textbf{irreducible} if the $\mathbb{Q}\left[G\right]$ module 
$$L_{\mathbb{Q}} := L \otimes_{\mathbb{Z}} \mathbb{Q}$$ 
is irreducible, i.e., if $L_{\mathbb{Q}}$ has no nontrivial proper $G$-invariant submodules.  
\end{defn}

\begin{defn} \label{irreducible matrix group def}
The subgroup $G \leq \operatorname{GL}_n\left(\mathbb{Z}\right)$ is said to be \textbf{irreducible} if the $G$-lattice $\mathbb{Z}^n$ is irreducible.
\end{defn}

The $G$-lattices $L$ and $L'$ are said to be \textit{$\mathbb{Z}$-equivalent} if $L \cong L'$ as $\mathbb{Z}\left[G\right]$-modules, and $L$ and $L'$ are said to be \textit{$\mathbb{Q}$-equivalent} if $L_{\mathbb{Q}} \cong L'_{\mathbb{Q}}$ as $\mathbb{Q}\left[G\right]$-modules.  This distinction between isomorphism of $G$-lattices over $\mathbb{Z}$ and over $\mathbb{Q}$ is subtle, but important.  We record the following useful observations about $\mathbb{Q}$-equivalence of $G$-lattices.

\begin{prop} \cite[\S 2]{PleskenPohst1} \cite[\S 1.2.1]{Lorenz} \label{Q-equivalence observations}
Let $G$ be a group and let $L$ be a $G$-lattice of rank $n$.
\begin{enumerate}
\item If $L' \subseteq L$ is a $G$-sublattice of rank $n$ (or, equivalently, of finite index in $L$), then $L$ and $L'$ are $\mathbb{Q}$-equivalent $G$-lattices.
\item Conversely, if $L_{\mathbb{Q}} \cong L'_{\mathbb{Q}}$, then replacing $L'$ with an isomorphic copy inside of $L_{\mathbb{Q}}$, we may assume that $L' \subseteq L$ and that the quotient module $L/L'$ is finite.
\end{enumerate}
\end{prop}

%

The invariant of $G$-lattices that we use to study representation dimensions of algebraic tori is the \textit{symmetric rank}.  This notion was introduced by MacDonald in \cite{MacDonald}, in which the author studied a related version of symmetric rank known as the \textit{symmetric $p$-rank} in order to study essential $p$-dimensions of normalizers of maximal tori.

\begin{defn} \label{symmetric rank def}
The \textbf{symmetric rank} of a $G$-lattice $L$, denoted $\operatorname{symrank}\left(L, G\right)$, is the minimal size of a $G$-stable generating set of $L$.  
\end{defn}

\begin{lem} \label{symmetric rank facts}
Let $n \in \mathbb{Z}^+$, let $G \leq \operatorname{GL}_n\left(\mathbb{Z}\right)$, and let $L \subseteq \mathbb{Z}^n$ be a $G$-lattice.
\begin{enumerate}
\item We have the bound $\operatorname{symrank}\left(L, G\right) \leq \left|G\right|\operatorname{rank}\left(L\right)$.
\item If $H \leq \operatorname{GL}_n\left(\mathbb{Z}\right)$ is conjugate over $\operatorname{GL}_n\left(\mathbb{Z}\right)$ to $G$, then there exists an $H$-lattice $M \subseteq \mathbb{Z}^n$ such that $M$ as an $H$-lattice is isomorphic to $L$ as a $G$-lattice.  In particular, $\operatorname{symrank}\left(M, H\right) = \operatorname{symrank}\left(L, G\right)$.
\item If $K \leq G$, then $L$ is a $K$-lattice with $\operatorname{symrank}\left(L, K\right) \leq \operatorname{symrank}\left(L, G\right)$.
\end{enumerate}
\end{lem}

\begin{proof}
For (a), let $\left\lbrace \ell_1, \dots, \ell_{\operatorname{rank}\left(L\right)}\right\rbrace$ be a generating set for $L$ as a lattice.  Observe that the set $S = \bigcup_{i=1}^{\operatorname{rank}\left(L\right)} G\ell_i$ is a $G$-stable generating set for $L$, and $\left|G\ell_i\right| \leq \left|G\right|$ for each $1 \leq i \leq \operatorname{rank}\left(L\right)$.  Thus, $S$ is a $G$-stable generating set for $L$ of size at most $\left|G\right| \operatorname{rank}\left(L\right)$.
\par
For (b), suppose $G = AHA^{-1}$ for $A \in \operatorname{GL}_n\left(\mathbb{Z}\right)$.  The matrix $A$ is a change of basis matrix from the standard basis for $\mathbb{Z}^n$ to some other basis $\mathcal{B}$ for $\mathbb{Z}^n$, so take $M$ to be the lattice $L$ written in the basis $\mathcal{B}$.
\par
For (c), it is clear that $L$ is a $K$-lattice.  The symmetric rank statement follows from the fact that the $G$-orbit of $v \in L$ is equal to the union of the orbits of $v$ under the actions of the cosets of $K$ in $G$.
\end{proof}

Note that any two $\mathbb{Z}$-equivalent $G$-lattices have the same symmetric rank by (b) of Lemma \ref{symmetric rank facts}, but $\mathbb{Q}$-equivalent $G$-lattices need not have equal symmetric ranks.
\par
Given an algebraic torus $T$, the \textit{character group} $X_T$ of $T$ is a $\mathscr{G}$-lattice for the absolute Galois group $\mathscr{G} := \operatorname{Gal}\left(k_s/k\right)$ (see, e.g., \cite[\S 7.3]{Waterhouse}).  Although $\mathscr{G}$ may be infinite, the following observation tells us that we can think of $\mathscr{G}$ as acting on $X_T$ via a finite group.

\begin{prop} \cite[\S 7.3]{Waterhouse} \label{absolute galois group acts as finite group}
Let $T$ be an algebraic torus with character group $X_T$.  Then, the action of $\mathscr{G}$ on $X_T$ factors through the finite group $\operatorname{Gal}\left(k'/k\right)$ for some finite Galois extension $k \subseteq k'$.
\end{prop}

Recall that an action of $\mathscr{G}$ on a set $X$ is said to be \textit{continuous} if $X$ is a union of $\mathscr{G}$-orbits where each orbit factors through $\operatorname{Gal}\left(k'/k\right)$ for some finite Galois extension $k \subseteq k'$.  The following theorem now makes precise the close relationship between $G$-lattices and algebraic tori.

\begin{thm} \cite[\S 7.3]{Waterhouse} \label{algebraic tori correspondence}
The functor $F$, from the category of algebraic tori to the category of finitely generated free Abelian groups on which $\mathscr{G}$ acts continuously, defined by $F\left(T\right) = X_T$ is an antiequivalence of categories.
\end{thm}

In particular, in the correspondence of Theorem \ref{algebraic tori correspondence}, irreducible tori correspond to irreducible $\mathscr{G}$-lattices.  
\par
We now state the connection between representation dimensions of algebraic tori and symmetric ranks of $G$-lattices.  See \cite[\S 4.4]{dissertation} for a proof.

\begin{thm} \label{symmetric rank theorem}
Let $T$ be an algebraic torus over the field $k$, let $X_T$ be the character group of $T$, and let $\mathscr{G}$ be the absolute Galois group of $k$.  Then,
$$\operatorname{rdim}\left(T\right) = \operatorname{symrank}\left(X_T, \mathscr{G}\right).$$
\end{thm}

\par
A direct consequence of Theorem \ref{symmetric rank theorem} and the fact that every finite group occurs (up to isomorphism) as the Galois group of some Galois field extension is the following.

\begin{cor} \label{symrank n = rdim n}
For all $n \in \mathbb{Z}^+$, we have
$$\operatorname{rdim}\left(n\right) = \max\left(\operatorname{symrank}\left(L, G\right)\right),$$
where the maximum is taken over all finite groups $G$ and all $G$-lattices of rank $n$.
\end{cor}

Even more explicitly, we obtain the following.

\begin{cor} \label{rdimn = max symrank of Z^n}
For all $n \in \mathbb{Z}^+$, we have
$$\operatorname{rdim}\left(n\right) = \max\left(\operatorname{symrank}\left(\mathbb{Z}^n, G\right)\right),$$
where the maximum is taken over one representative $G$ from each $\mathbb{Z}$-conjugacy class of maximal (with respect to inclusion) finite subgroups of $\operatorname{GL}_n\left(\mathbb{Z}\right)$.
\end{cor}



We similarly have the following for $\operatorname{rdim}_{\operatorname{irr}}\left(n\right)$.

\begin{prop} \label{reduction of symrankirr}
For all $n \in \mathbb{Z}^+$, we have
$$\operatorname{rdim}_{\operatorname{irr}}\left(n\right) = \max\left(\operatorname{symrank}\left(\mathbb{Z}^n, G\right)\right),$$
where the maximum is taken over one representative $G$ from each $\operatorname{GL}_n\left(\mathbb{Z}\right)$-conjugacy class of irreducible, maximal, finite subgroups of $\operatorname{GL}_n\left(\mathbb{Z}\right)$.
\end{prop}

This motivates the study of finite subgroups of $\operatorname{GL}_n\left(\mathbb{Z}\right)$.  One of the most famous results on this subject is the following theorem.

\begin{thm} [\textbf{Jordan-Zassenhaus Theorem}] \cite[Thm. 2.6]{Reiner} \label{Jordan-Zassenhaus Theorem}
For all $n \in \mathbb{Z}^+$, there are a finite number of $\mathbb{Z}$-conjugacy classes of finite subgroups of $\operatorname{GL}_n\left(\mathbb{Z}\right)$.
\end{thm}

An important consequence of the Jordan-Zassenhaus Theorem is that the quantities $\operatorname{rdim}\left(n\right)$ and $\operatorname{rdim}_{\operatorname{irr}}\left(n\right)$ are well-defined.

\section{Root Systems and Weyl Groups} \label{root system section}

Some of the most important examples of $G$-lattices come from root systems and their Weyl groups.  In this section, we compute the symmetric ranks of all $W$-lattices when $W$ is the Weyl group of an irreducible root system, yielding lower bounds on $\operatorname{rdim}\left(n\right)$.  We will then show that these bounds are realized over any number field.  Theorem \ref{root systems theorem}, along with the overall ideas in this section, are very similar to work done by Lemire to obtain upper bounds on essential dimensions of certain linear algebraic groups \cite{Lemire}.
\par
We will follow the notation and conventions of \cite[Ch. 3]{Humphreys}.
\begin{itemize} \label{root system notation}
\item $\Phi$ denotes an irreducible root system of rank $n \in \mathbb{Z}^+$ in a Euclidean space $E$ with inner product $\left(\cdot, \cdot\right)$.
\item $W\left(\Phi\right)$, or simply $W$, is the Weyl group of $\Phi$.
\item $\Lambda_r$ is the root lattice associated to $\Phi$, with base $\left\lbrace \alpha_1, \dots, \alpha_n\right\rbrace$.
\item For $1 \leq i \leq n$, we write $\sigma_i:  E \to E$ for the reflection fixing the hyperplane orthogonal to $\alpha_i$, i.e., the hyperplane $\left\lbrace \beta \in E \ | \ \left(\beta, \alpha_i\right) = 0\right\rbrace$.
\item $\Lambda$ is the weight lattice associated to $\Phi$, with base $\left\lbrace \lambda_1, \dots, \lambda_n\right\rbrace$ given as follows:  For $1 \leq i \leq n$, we define $\lambda_i \in E$ to be the unique vector satisfying $\frac{2\left(\lambda_i, \alpha_j\right)}{\left(\alpha_j, \alpha_j\right)} = \delta_{i,j}$ for all $1 \leq j \leq n$, where $\delta_{i,j}$ is the Kronecker delta.
\item We denote by $\Lambda^{+i}$ an intermediate lattice $\Lambda_r \subset \Lambda^{+i} \subset \Lambda$, with both inclusions proper and $\left[\Lambda : \Lambda^{+i}\right] = i$.  In the case $\Phi = \mathsf{D}_n$ with $n \geq 4$ even, there are two intermediate lattices with index $2$ in $\Lambda$, each containing $\lambda_{n-1}$ or $\lambda_n$ but not both.  We will denote these intermediate lattices by $\Lambda^{+2}_i$, with $\lambda_i \in \Lambda^{+2}_i$ for each $i \in \left\lbrace n - 1, n\right\rbrace$.
\end{itemize}

Computing symmetric ranks of root lattices is quite straightforward.

\begin{prop} \cite[\S 1]{MacDonald} \label{symranks of root lattices}
Let $\Phi$ be an irreducible root system with Weyl group $W$ and root lattice $\Lambda_r$.  Then, the symmetric rank $\operatorname{symrank}\left(\Lambda_r, W\right)$ is equal to the number of short roots in $\Phi$.
\end{prop}

The following standard facts help us compute symmetric ranks of other $W$-lattices.

\begin{prop} \cite[\S 13.1]{Humphreys} \label{Weyl group action on weight lattice}
For all $1 \leq i, j \leq n$, we have
$$\sigma_i\left(\lambda_j\right) = \lambda_j - \delta_{i,j}\alpha_i.$$
In particular, for each $1 \leq i \leq n$, the reflection $\sigma_i$ fixes all of the $\lambda_j$'s except for $\lambda_i$.
\end{prop}

\begin{lem} \label{stabilizers in Weyl groups}
For any $v = \sum_{i=1}^n v_i\lambda_i \in \Lambda$, the stabilizer of $v$ in $W$ is given by 
$$\operatorname{stab}_{W}\left(v\right) = \left\langle \sigma_i \ | \ v_i = 0\right\rangle \leq W.$$
\end{lem}

\begin{proof}
By Proposition \ref{Weyl group action on weight lattice}, we have $\left\langle \sigma_i \ | \ v_i = 0\right\rangle \subseteq \operatorname{stab}_{W}\left(v\right)$.  The reverse inclusion follows from \cite[Lem. 1]{Carter}.
\end{proof}

We now present the main result of this section.

\begin{thm} \label{root systems theorem}
Let $\Phi$ be an irreducible root system of rank $n$ with base $\left\lbrace \alpha_1, \dots, \alpha_n\right\rbrace$ and Weyl group $W$, let $\Lambda_r$ be the root lattice, and let $\Lambda$ be the weight lattice with basis $\left\lbrace \lambda_1, \dots, \lambda_n\right\rbrace$ as defined above.  Then, we have the values in Table \ref{root system theorem table} for the symmetric ranks of all $W$-lattices for all choices of $\Phi$ and $n$.  In particular, the orbit $Wv$ is a particular $W$-stable generating set of minimal size for the given $v \in \Lambda$.

\begin{table}[h]
\caption[Symmetric ranks of $W$-lattices, where $W$ is the Weyl group of an irreducible root system.]{Symmetric ranks of $W$-lattices, where $W$ is the Weyl group of an irreducible root system.  See p. \pageref{root system notation} and Theorem \ref{root systems theorem} for notation.}
\label{root system theorem table}
\begin{center}
\begin{tabular}{|c|c|c|c|}
\hline
Root System $\Phi$ & Lattice $L$ & $\operatorname{symrank}\left(L, W\right)$ & Generator $v$ \\\hline $\mathsf{A}_n$, $n \geq 1$ & $\Lambda$ & $n+1$ & $\lambda_1$ \\\hline & & & \\ $\mathsf{A}_n$, $n \geq 1$ & $\Lambda^{+d}$, $d \ | \ n + 1$ & $\binom{n+1}{d}$ & $\lambda_d$ \\ & & & \\\hline $\mathsf{A}_n$, $n \geq 1$ & $\Lambda_r$ & $n\left(n+1\right)$ & $\alpha_1$ \\\hline $\mathsf{B}_n$, $n \geq 2$ & $\Lambda$ & $2^n$ & $\lambda_n$ \\\hline $\mathsf{B}_n$, $n \geq 2$ & $\Lambda_r$ & $2n$ & $\alpha_n$ \\\hline $\mathsf{C}_n$, $n \geq 3$ & $\Lambda$ & $2n$ & $\lambda_1$ \\\hline $\mathsf{C}_n$, $n \geq 3$ & $\Lambda_r$ & $2n\left(n-1\right)$ & $\alpha_1$ \\\hline $\mathsf{D}_n$, $n \geq 4$ even & $\Lambda$ & $2n$ & $\lambda_1$ \\\hline & & & \\ $\mathsf{D}_n$, $n \geq 4$ even & $\Lambda^{+2}_{n-1}$ & $2^{n-1}$ & $\lambda_{n-1}$ \\ & & & \\\hline & & & \\ $\mathsf{D}_n$, $n \geq 4$ even & $\Lambda^{+2}_n$ & $2^{n-1}$ & $\lambda_n$ \\ & & & \\\hline $\mathsf{D}_n$, $n \geq 4$ & $\Lambda_r$ & $2n\left(n-1\right)$ & $\alpha_1$ \\\hline $\mathsf{D}_n$, $n \geq 5$ odd & $\Lambda$ & $2^{n-1}$ & $\lambda_n$ \\\hline & & & \\ $\mathsf{D}_n$, $n \geq 5$ odd & $\Lambda^{+2}$ & $2n$ & $\lambda_1$ \\ & & & \\\hline $\mathsf{E}_6$ & $\Lambda$ & $27$ & $\lambda_1$ \\\hline $\mathsf{E}_6$ & $\Lambda_r$ & $72$ & $\alpha_1$ \\\hline $\mathsf{E}_7$ & $\Lambda$ & $56$ & $\lambda_7$ \\\hline $\mathsf{E}_7$ & $\Lambda_r$ & $126$ & $\alpha_1$ \\\hline $\mathsf{E}_8$ & $\Lambda = \Lambda_r$ & $240$ & $\alpha_1$ \\\hline $\mathsf{F}_4$ & $\Lambda = \Lambda_r$ & $24$ & $\alpha_1$ \\\hline $\mathsf{G}_2$ & $\Lambda = \Lambda_r$ & $6$ & $\alpha_1$ \\\hline
\end{tabular}
\end{center}
\end{table}
\end{thm}

\begin{proof}
Let $C$ be the Cartan matrix of $\Phi$, as given in \cite[\S 11.4]{Humphreys}.  If $\det\left(C\right) = 1$, then $\Lambda_r = \Lambda$ and we are done by Proposition \ref{symranks of root lattices}.
\par
Otherwise, we identify $\Lambda$ with $\mathbb{Z}^n$ by identifying each $\lambda_i$ with the standard basis vector $e_i$ and proceed as follows.  Using the Smith normal form, there exist $P, Q \in \operatorname{GL}_n\left(\mathbb{Z}\right)$ and an $n \times n$ integral diagonal matrix $D$ such that $PCQ = D$.  Note that the choices of $P$ and $Q$ are not canonical; our choices can be found in \cite[Appendix A]{dissertation}.  As $C$ is the change of basis matrix from the $\lambda_i$'s to the $\alpha_i$'s \cite[\S 13.1]{Humphreys}, we have $v \in \Lambda$ is an element of $\Lambda_r$ if and only if $Cw = v$ for some $w \in \mathbb{Z}^n$.  Since $Q \in \operatorname{GL}_n\left(\mathbb{Z}\right)$, this is equivalent to the existence of some $u \in \mathbb{Z}^n$ such that $CQu = P^{-1}Du = v$.  This equation provides conditions (via a system of linear equations) for determining the smallest (with respect to inclusion) $W$-sublattice of $\Lambda$ containing a given $v \in \Lambda$.  For each choice of $\Phi$ and $n$, the matrix $P^{-1}D$ is given in Table \ref{root system smith normal form table 1}.  From there, it is a standard fact (see, e.g., \cite[Lemma 10.4.B]{Humphreys}) that the $W$-orbit of any nonzero $v \in \Lambda$ has full rank.  Therefore, if $L \subseteq \Lambda$ is the smallest $W$-sublattice of $\Lambda$ containing $v$, then the lattice spanned by $Wv$ must be $L$.  We then use Lemma \ref{stabilizers in Weyl groups} and information about the Weyl groups of root systems (see, e.g., \cite[\S 12.2]{Humphreys}) to compute the sizes of these orbits.
\end{proof}

\begin{table}[h]
\centering
\caption[Matrices used in the proof of Theorem \ref{root systems theorem}.]{The $n \times n$ matrices $P^{-1}D$ used in the proof of Theorem \ref{root systems theorem}.  We show only the bottom row(s) of $P^{-1}D$; the remaining rows of $P^{-1}D$ are equal to those of the $n \times n$ identity matrix.}
\label{root system smith normal form table 1}
\begin{tabular}{|c|c|}
\hline
$\Phi$ & Bottom Row(s) of $P^{-1}D$ \\\hline 
& \\
$\mathsf{A}_n$, $n \geq 2$ &  
$\bmatrix
1 & 2 & 3 & \cdots & n - 1 & n + 1
\endbmatrix$ \\
& \\\hline
& \\
$\mathsf{B}_n$, $n \geq 2$ & 
$\bmatrix
0 & 0 & \cdots & 0 & 2
\endbmatrix$
\\
& \\\hline
& \\
$\mathsf{C}_n$, $n \geq 3$ odd &
$\bmatrix
1 & 0 & 1 & 0 & \cdots & 1 & 0 & 2
\endbmatrix$
\\
& \\\hline
& \\
$\mathsf{C}_n$, $n \geq 4$ even &
$\bmatrix
1 & 0 & 1 & 0 & 1 & 0 & \cdots & 1 & 0 & 0 & -2 \\
0 & 0 & 0 & 0 & 0 & 0 & \cdots & 0 & 0 & 1 & 2
\endbmatrix$ \\
& \\\hline
& \\
$\mathsf{D}_n$, $n \geq 4$ even &
$\bmatrix
1 & 0 & 1 & 0 & 1 & 0 & \cdots & 1 & 0 & 2 & 0\\
1 & 0 & 1 & 0 & 1 & 0 & \cdots & 1 & 0 & 0 & 2
\endbmatrix$
\\
& \\\hline
& \\
$\mathsf{D}_n$, $n \geq 5$ odd &
$\bmatrix
2 & 0 & 2 & 0 & \cdots & 2 & 1 & 4
\endbmatrix$
\\
& \\\hline
& \\
$\mathsf{E}_6$ & 
$\bmatrix 
1 & 0 & 2 & 0 & 1 & 3
\endbmatrix$
\\
& 
\\\hline
& \\
$\mathsf{E}_7$ & 
$\bmatrix 
0 & 1 & 0 & 0 & 1 & 0 & 2
\endbmatrix$
\\
& 
\\\hline
\end{tabular}
\end{table}

We now have the following immediate corollary of Theorem \ref{root systems theorem}.

\begin{cor} \label{lower bound on rdim n}
In Table \ref{lower bound on rdim n table}, we have lower bounds on $\operatorname{rdim}\left(n\right)$ for all $n \in \mathbb{Z}^+$, along with a $G$-lattice realizing each bound.
\begin{table}[h]
\caption[Lower bounds on $\operatorname{rdim}\left(n\right)$ and $\operatorname{symrank}_{\operatorname{irr}}\left(n\right)$ for all $n \in \mathbb{Z}^+$.]{Lower bounds on $\operatorname{rdim}\left(n\right)$ and $\operatorname{symrank}_{\operatorname{irr}}\left(n\right)$ for all $n \in \mathbb{Z}^+$, along with a $G$-lattice realizing each bound.  For a root system $\Phi$, the group $W\left(\Phi\right)$ is the Weyl group, $\Lambda_r$ is the root lattice, and $\Lambda$ is the weight lattice.}
\label{lower bound on rdim n table}
\begin{center}\begin{tabular}{|c|c|c|} \hline $n$ & $\operatorname{rdim}\left(n\right) \geq$ & $G$-Lattice \\\hline $1$ & $2$ & $\left(\Lambda_r, W\left(\mathsf{A}_1\right)\right)$ \\\hline $2$ & $6$ & $\left(\Lambda_r, W\left(\mathsf{G}_2\right)\right)$ \\\hline $3$ & $12$ & $\left(\Lambda_r, W\left(\mathsf{A}_3\right)\right)$ \\\hline $4$ & $24$ & $\left(\Lambda_r, W\left(\mathsf{C}_4\right)\right)$ \\\hline $5$ & $40$ & $\left(\Lambda_r, W\left(D_5\right)\right)$ \\\hline $6$ & $72$ & $\left(\Lambda_r, W\left(\mathsf{E}_6\right)\right)$ \\\hline $\geq 7$ & $2^n$ & $\left(\Lambda, W\left(\mathsf{B}_n\right)\right)$ \\\hline
\end{tabular}
\end{center}
\end{table}
\end{cor}

We now show that these bounds are realized over any number field.  

\begin{thm} \label{realization over number fields}
For any number field $K$ and any $n \in \mathbb{Z}^+$, there exists an algebraic torus $T$ over $K$ of dimension $n$ with $\operatorname{rdim}\left(T\right)$ equal to the bound given in Corollary \ref{lower bound on rdim n}.
\end{thm}

\begin{proof}
The key idea is the following result of \cite{Nuzhin}:  For all irreducible root systems $\Phi$ except $\mathsf{F}_4$, there exists a regular Galois extension of the function field $\mathbb{Q}\left(t\right)$ (where $t$ is an indeterminate) with Galois group isomorphic to the Weyl group $W\left(\Phi\right)$.  As $W\left(\mathsf{F}_4\right)$ does not occur in Table \ref{lower bound on rdim n table}, we need not consider this case (although we note that $W\left(\mathsf{F}_4\right)$ is solvable and, hence, is realized as a Galois group over $\mathbb{Q}$ \cite[p. 171]{Nuzhin}).
\par
Now, let $K$ be a number field and let $W$ be the Weyl group of an irreducible root system besides $\mathsf{F}_4$.  By \cite[Cor. 3.2.4]{Jensen}, $K$ is Hilbertian.   Since $W$ is realizable as a Galois group over $\mathbb{Q}\left(t\right)$ by the paragraph above, Proposition 3.3.6 of \cite{Jensen} implies that $W$ is realizable as a Galois group over $K\left(t\right)$.  Result 3.3.4 of \cite{Jensen} now asserts that $W$ occurs as a Galois group over $K$, completing the proof.
\end{proof}

\section{Irreducible $G$-Lattices of Rank $n \in \left\lbrace 1, 2, \dots, 10, 11, 13, 17, 19, 23\right\rbrace$} \label{small dimensions chapter}

Now, we compute $\operatorname{rdim}_{\operatorname{irr}}\left(n\right)$ for $n \in \left\lbrace 1, 2, \dots, 10, 11, 13, 17, 19, 23\right\rbrace$.  In particular, we will show that the lower bounds given in Corollary \ref{lower bound on rdim n} are sharp by finding upper bounds on the symmetric ranks of all irreducible $G$-lattices in these dimensions. These dimensions are special because of Theorem \ref{finite subgroups of GLn(Z)} below.  We use the common abbreviation i.m.f. for ``irreducible, maximal, finite."

\begin{thm} \cite{Voskresenskii} \cite{Tahara} \cite{Dade} \cite{Wondratschek} \cite{Ryskov} \cite{PleskenPohst1} \cite{PleskenPohst2} \cite{PleskenPohst3} \cite{PleskenPohst4} \cite{PleskenPohst5} \cite{Plesken} \cite{Souvignier} \label{finite subgroups of GLn(Z)}
Explicit representatives of the $\mathbb{Z}$-conjugacy classes of i.m.f. subgroups of $\operatorname{GL}_n\left(\mathbb{Z}\right)$ are known for $1 \leq n \leq 11$ and $n = 13, 17, 19, 23$.
\end{thm}


Representatives from Theorem \ref{finite subgroups of GLn(Z)} are stored in the computer algebra system GAP \cite{GAP}, along with other information about the groups.  Each group is identified via an ordered triple $\left(d, q, z\right)$, where $d$ is the dimension, $q$ is the numerical label of the $\operatorname{GL}_n\left(\mathbb{Q}\right)$-class, and $z$ is the numerical label of the $\operatorname{GL}_n\left(\mathbb{Z}\right)$-class within the $\operatorname{GL}_n\left(\mathbb{Q}\right)$-class.  For more details, please see 50.7 of the online GAP manual.
\par
We have the following useful characterization of all i.m.f. subgroups of $\operatorname{GL}_n\left(\mathbb{Z}\right)$.

\begin{prop} \cite[p. 536]{PleskenPohst1} \label{imf subgroups fix quadratic forms}
For every i.m.f. subgroup $G < \operatorname{GL}_n\left(\mathbb{Z}\right)$, there exists a (not necessarily unique) positive definite $n \times n$ integral matrix $X$ such that $G$ is the full group of $\mathbb{Z}$-automorphisms of $X$, i.e., $G = \operatorname{Aut}_{\mathbb{Z}}\left(X\right) = \left\lbrace M \in \operatorname{GL}_n\left(\mathbb{Z}\right) \ | \ M^TXM = X\right\rbrace$.
\end{prop}

Along with other data about the group, GAP stores such a matrix $X$ for each i.m.f. subgroup.  

\begin{defn} \label{norm and theta series def}
Given a positive definite $n \times n$ integral matrix $X$ and a column vector $v \in \mathbb{Z}^n$, we call the integer $v^TXv$ the \textbf{norm} of $v$ under $X$.  The \textbf{$\Theta$-series} of $X$ is the power series $\Theta_X\left(q\right) = \sum_{i=0}^{\infty} N_iq^i$ with coefficients given by
$$N_i := \left|\left\lbrace v \in \mathbb{Z}^n \ | \ v^TXv = i\right\rbrace\right|.$$
\end{defn}

We now obtain a bound on $\operatorname{symrank}\left(\mathbb{Z}^n, G\right)$ in terms of the $\Theta$-series coefficients.

\begin{lem} \label{diagonal norms form a generating set}
Let $G < \operatorname{GL}_n\left(\mathbb{Z}\right)$ be an i.m.f. subgroup and suppose that $G = \operatorname{Aut}_{\mathbb{Z}}\left(X\right)$ for the positive definite $n \times n$ integral matrix $X$.  Write $D = \left\lbrace X_{i,i} \ | \ 1 \leq i \leq n\right\rbrace$ for the set of diagonal entries of $X$ and $\Theta_X\left(q\right) = \sum_{i=0}^{\infty} N_iq^i$ for the $\Theta$-series of $X$.  Then,
$$\operatorname{symrank}\left(\mathbb{Z}^n, G\right) \leq \sum_{i \in D} N_i.$$
\end{lem}

\begin{proof}
Observe that the set $\left\lbrace v \in \mathbb{Z}^n \ | \ v^TXv \in D\right\rbrace$ is $G$-stable and contains the standard basis vectors for $\mathbb{Z}^n$.
\end{proof}

Although the bound given in Lemma \ref{diagonal norms form a generating set} is rather coarse in general, in most cases it is sufficient to prove the sharpness of the claimed bounds in dimensions $1, 2, \dots, 10, 11, 13, 17, 19, 23$.

\begin{thm} \label{low dimensions theorem}
In Table \ref{low dimensions theorem}, we have exact values of $\operatorname{rdim}_{\operatorname{irr}}\left(n\right)$, along with a $G$-lattice realizing each value.
\begin{table}[h]
\centering
\caption{Exact values of $\operatorname{rdim}_{\operatorname{irr}}\left(n\right)$ for dimensions $n \in \left\lbrace 1, 2, \dots, 10, 11, 13, 17, 19, 23\right\rbrace$.}
\label{low dimensions theorem table}
\begin{tabular}{|c|c|c|} \hline \textbf{$n$} & \textbf{$\operatorname{rdim}_{\operatorname{irr}}\left(n\right)$} & $G$-Lattice \\\hline $1$ & $2$ & $\left(\Lambda_r, W\left(\mathsf{A}_1\right)\right)$ \\\hline $2$ & $6$ & $\left(\Lambda_r, W\left(\mathsf{G}_2\right)\right)$ \\\hline $3$ & $12$ & $\left(\Lambda_r, W\left(\mathsf{A}_3\right)\right)$ \\\hline $4$ & $24$ & $\left(\Lambda_r, W\left(\mathsf{C}_4\right)\right)$ \\\hline $5$ & $40$ & $\left(\Lambda_r, W\left(\mathsf{D}_5\right)\right)$ \\\hline $6$ & $72$ & $\left(\Lambda_r, W\left(\mathsf{E}_6\right)\right)$ \\\hline $7, 8, 9, 10, 11, 13, 17, 19, 23$ & $2^n$ & $\left(\Lambda, W\left(\mathsf{B}_n\right)\right)$ \\\hline
\end{tabular}
\end{table}
\end{thm}

\begin{proof}
By Corollary \ref{lower bound on rdim n}, we need only to show that the values in Theorem \ref{low dimensions theorem} are upper bounds on $\operatorname{rdim}_{\operatorname{irr}}\left(n\right)$ for the given values of $n$.  To do this, we use a computer algebra system to apply Lemma \ref{diagonal norms form a generating set} to the representative from each $\operatorname{GL}_n\left(\mathbb{Z}\right)$-conjugacy class of i.m.f. subgroups of $\operatorname{GL}_n\left(\mathbb{Z}\right)$ stored in GAP; see \cite[Appendix B]{dissertation} for our code and the output.  We now address the groups for which Lemma \ref{diagonal norms form a generating set} does not yield the claimed bound.  
\par
The groups given by the following ordered triples are all isomorphic to $C_2 \times W\left(\mathsf{D}_n\right)$, where $C_2$ is the cyclic group of order $2$:  $\left(d, 1, 3\right)$ for $5 \leq d \leq 10$, $\left(19, 1, 2\right)$, and $\left(23, 2, 3\right)$.  In each of these cases, by Theorem \ref{root systems theorem}, we have $\operatorname{symrank}\left(\mathbb{Z}^n, G\right) \leq 2 \cdot 2^{n-1} = 2^n$.
\par
This leaves only the $23$-dimensional group $G$ given by the ordered triple $\left(23, 3, 2\right)$ to consider.  For this case, we first obtain that $G = \left\langle M_1, M_2\right\rangle$ for $M_1, M_2 \in \operatorname{GL}_{23}\left(\mathbb{Z}\right)$.  (See \cite[p. 62-63]{dissertation} for $M_1$ and $M_2$.)  The minimal (nonzero) norm under the corresponding matrix $X$ is $4$, and there are $93150$ vectors in $\mathbb{Z}^{23}$ of norm $4$ under $X$, one of which is the standard basis vector $e_2 \in \mathbb{Z}^{23}$.  By computing the union of the orbits of $e_2$ under the subgroups $\left\langle M_1\right\rangle, \left\langle M_2\right\rangle$, and $\left\langle M_1M_2\right\rangle$, we see that the orbit $Ge_2$ spans $\mathbb{Z}^{23}$ as a lattice.  As $Ge_2$ is a subset of the set of elements of norm $4$ under $X$, it follows that $\operatorname{symrank}\left(\mathbb{Z}^{23}, G\right) \leq 93150 < 2^{23}$, as desired.
\end{proof}

\begin{rem} \label{Q-class remark}
For all dimensions $1 \leq n \leq 31$, $\mathbb{Q}$-class representatives of i.m.f. subgroups of $\operatorname{GL}_n\left(\mathbb{Z}\right)$ are known and stored in GAP.  However, since $\mathbb{Q}$-equivalent $G$-lattices can have different symmetric ranks, the techniques applied in this section cannot (at present) be directly applied to those dimensions not addressed in Theorem \ref{low dimensions theorem}.
\end{rem}

\section{Irreducible $G$-Lattices of Prime Rank} \label{prime dimensions chapter}

In this section and the next, we study irreducible $G$-lattices in prime dimensions to prove Theorem \ref{asymptotic theorem} and (b) of Theorem \ref{symrankirr theorem 1}.  Prime dimensions are especially nice because of the following fact.

\begin{prop} \cite[\S 1]{PleskenPohst1} \label{irreducibility over Z and C}
If $\Delta:  G \to \operatorname{GL}_p\left(\mathbb{Z}\right)$ is an irreducible representation of a finite group $G$ of prime degree $p$, then $\Delta$ is irreducible over $\mathbb{Z}$ if and only if $\Delta$ is irreducible over $\mathbb{C}$.
\end{prop}

In odd prime dimensions $p$, the following proposition classifies irreducible finite subgroups $G < \operatorname{GL}_p\left(\mathbb{Z}\right)$ by their maximal normal 2-subgroups.  We first introduce the following notation for a finite subgroup $G < \operatorname{GL}_n\left(\mathbb{Z}\right)$:
\begin{itemize}
\item $G^+ := G \cap \operatorname{SL}_n\left(\mathbb{Z}\right)$.
\item We denote by $O_2\left(G\right)$ the maximal normal $2$-subgroup of $G$.
\end{itemize}

\begin{prop} \cite[Prop. II.2]{Plesken} \label{Plesken prop}
Let $p > 2$ be a prime integer, let $G < \operatorname{GL}_p\left(\mathbb{Z}\right)$ be an irreducible finite group, and let $d$ be the multiplicative order of $2$ modulo $p$.  Then, one of the following occurs:
\begin{enumerate}
\item $O_2\left(G^+\right)$ is elementary abelian of order $2^{d\ell}$ for some integer $1 \leq \ell \leq \frac{p-1}{d}$.  In this case, $G^+$ is conjugate under $\operatorname{GL}_p\left(\mathbb{Q}\right)$ to a group of integral monomial matrices (i.e., signed permutation matrices).  Moreover, $G^+$ contains an irreducible subgroup which is an extension of an elementary abelian $2$-group of order $2^d$ by a cyclic group of order $p$.
\item $O_2\left(G^+\right) = 1$.  In this case, $G^+$ has a unique minimal normal subgroup $N \neq 1$ (possibly $N = G^+$), and $N$ is nonabelian simple, the centralizer $C_{G^+}\left(N\right) = 1$, and $N$ is irreducible as a matrix group.
\end{enumerate}  
\end{prop}

We will use Proposition \ref{Plesken prop} to consider two separate cases for irreducible finite subgroups of $\operatorname{GL}_p\left(\mathbb{Z}\right)$:  Those satisfying (a), which we focus on in this section; and those satisfying (b), which we will address in Section \ref{almost simple groups chapter}.
\par
We denote by $\operatorname{Mon}_n\left(\mathbb{Z}\right)$ the group of all $n \times n$ monomial matrices and by $\operatorname{D}_n\left(\mathbb{Z}\right)$ the subgroup of diagonal matrices.  Note that $\operatorname{D}_n\left(\mathbb{Z}\right) \cong C_2^n$; in particular, 
$$\operatorname{Mon}_n\left(\mathbb{Z}\right) \cong \operatorname{D}_n\left(\mathbb{Z}\right) \rtimes S_n \cong C_2^n \rtimes S_n,$$
where $C_2$ is the cyclic group of order $2$ and $S_n$ is the symmetric group on $n$ letters.
\par
Unlike in those dimensions addressed in Chapter \ref{small dimensions chapter}, in prime dimensions $p \geq 29$ we have only a description of the $\operatorname{GL}_p\left(\mathbb{Q}\right)$-conjugacy classes of irreducible, maximal, finite (i.m.f.) subgroups as opposed to the $\operatorname{GL}_p\left(\mathbb{Z}\right)$-conjugacy classes.  Thus, we use the following slightly different strategy in this section.

\begin{prop} \label{reduction to monomial matrices}
Let $G < \operatorname{GL}_p\left(\mathbb{Z}\right)$ be an irreducible finite group satisfying (a) of Proposition \ref{Plesken prop}.  Then, we have
$$\operatorname{symrank}\left(\mathbb{Z}^p, G\right) \leq \max_{H, L}\left(\operatorname{symrank}\left(L, H\right)\right),$$
where the maximum is taken over all irreducible subgroups $H \leq \operatorname{Mon}_p\left(\mathbb{Z}\right)$ such that $-I_p \in H$ and all $H$-lattices $L \subseteq \mathbb{Z}^p$.
\end{prop}

\begin{proof}
This follows from Proposition \ref{Q-equivalence observations} and (a) of Proposition \ref{Plesken prop}.  Note that we may restrict our attention to those groups satisfying $-I_p \in H$ because any group not satisfying this condition is certainly not a maximal subgroup of $\operatorname{GL}_p\left(\mathbb{Z}\right)$ up to $\mathbb{Z}$-conjugacy.
\end{proof}

\begin{observation} \label{when -I_n in G}
If $-I_n \in G \leq \operatorname{GL}_n\left(\mathbb{Z}\right)$ for some odd $n \in \mathbb{Z}^+$, then:
\begin{enumerate}
\item $G = -G$.
\item $\left|G\right| = 2\left|G^+\right|$.
\item $O_2\left(G\right) \cong O_2\left(G^+\right) \times C_2$.
\end{enumerate}
\end{observation}


Unless stated otherwise, throughout the remainder of this section, $G \leq \operatorname{Mon}_n\left(\mathbb{Z}\right)$ is an irreducible group satisfying the conditions that $-I_n \in G$, that $\left|O_2\left(G^+\right)\right| > 1$, and that $\pi\left(G\right) \leq S_n$ contains an $n$-cycle (because of the final sentence in (a) of Proposition \ref{Plesken prop}), where $\pi:  G \to S_n$ is the projection map.  
\par
We now classify the possibilities for $O_2\left(G\right)$.  First, we make an observation about $2$-subgroups of $S_n$ for odd $n > 2$.

\begin{lem} \label{p-cycle has no fixed points}
If $\sigma \in S_n$ is an $n$-cycle for $n \geq 3$ odd, then the action of conjugation by $\sigma$ on the set of 2-subgroups of $S_n$ has no nontrivial fixed points.
\end{lem}

\begin{proof}
By \cite[Thm. 4.2.3]{Covello}, we have the following description of the Sylow-2 subgroups of $S_n$:  Writing $n$ in its base-$2$ representation $n = \sum_{i=0}^s a_i2^i$ (each $a_i \in \left\lbrace 0, 1\right\rbrace$ and $a_s = 1$), the Sylow $2$-subgroups of $S_n$ are isomorphic to the direct products of the Sylow 2-subgroups of $S_{2^i}$ for $0 \leq i \leq s$ with $a_i = 1$.  Since $n$ is odd, any nontrivial 2-subgroup $S \leq S_n$ therefore leaves fixed some $1 \leq j \leq n$, but some element of $\sigma S \sigma^{-1}$ moves $j$ since $\sigma$ contains an $n$-cycle.  Thus, $\sigma S \sigma^{-1} \neq S$.
\end{proof}

\begin{lem} \label{O_2(G) consists of diagonal matrices}
Let $n > 2$ be odd and let $G \leq \operatorname{Mon}_n\left(\mathbb{Z}\right)$ be a group of monomial matrices such that $\pi\left(G\right)$ contains an $n$-cycle.  Then, $O_2\left(G\right) \leq \operatorname{D}_n\left(\mathbb{Z}\right)$ (or, equivalently, $O_2\left(G\right) = G \cap \operatorname{D}_n\left(\mathbb{Z}\right)$) and $G \cong O_2\left(G\right) \rtimes \pi\left(G\right)$.
\end{lem}

\begin{proof}
We prove that $O_2\left(G\right) \leq \operatorname{D}_n\left(\mathbb{Z}\right)$; the subsequent statement then follows.  View $G$ as a subgroup of $\operatorname{D}_n\left(\mathbb{Z}\right) \rtimes S_n$, and observe that $\pi\left(O_2\left(G\right)\right) \trianglelefteq \pi\left(G\right)$ is a normal 2-subgroup.  Since $\pi\left(G\right)$ contains an $n$-cycle, Lemma \ref{p-cycle has no fixed points} implies that $\pi\left(O_2\left(G\right)\right)$ must be trivial.
\end{proof}

We now have the following corollary of Lemma \ref{O_2(G) consists of diagonal matrices}.  Notationally, for a prime power $q$, we write $\mathbb{F}_q$ for the field with $q$ elements.  When $q$ is prime, we take $\mathbb{F}_q$ to be the set $\left\lbrace 0, 1, \dots, q - 1\right\rbrace$ under the operations of addition and multiplication modulo $q$

\begin{cor} \label{O_2(G) is a pi(G) stable subspace}
Under the conditions of Lemma \ref{O_2(G) consists of diagonal matrices}, $O_2\left(G\right)$ must correspond to a $\pi\left(G\right)$-stable subspace of $C_2^p \cong \mathbb{F}_2^p$, where $\pi\left(G\right) \leq S_p$ acts by shuffling coordinates.
\end{cor}

For $\sigma \in \pi\left(G\right)$, we will write $M_{\sigma}$ for the permutation matrix in $\operatorname{GL}_n\left(\mathbb{Z}\right)$ corresponding to $\sigma$.  At the level of matrices, the explicit action of $\pi\left(G\right)$ on $O_2\left(G\right) = G \cap \operatorname{D}_p\left(\mathbb{Z}\right)$ is as follows:  For $D \in O_2\left(G\right)$, the permutation $\sigma \in \pi\left(G\right)$ acts on $D$ via 
$$D \mapsto M_{\sigma}DM_{\sigma}^{-1}.$$  
This has the effect of taking the $\left(i, i\right)$ entry in $D$ to position $\left(\sigma\left(i\right), \sigma\left(i\right)\right)$.
\par
Corollary \ref{O_2(G) is a pi(G) stable subspace} motivates the study of subspaces of $\mathbb{F}_2^p$ that are stable under certain subgroups of $S_p$.  In particular, as the cyclic group $C_p \leq \pi\left(G\right)$ by assumption, the set of $\pi\left(G\right)$-stable subspaces of $\mathbb{F}_2^p$ is a subset of the set of $C_p$-stable subspaces of $\mathbb{F}_2^p$.  We will write $\overline{k}$ for the algebraic closure of the field $k$.

\begin{lem} \label{factorization of x^p - 1 lemma}
Let $p > 2$ be prime, let $d$ be the multiplicative order of $2$ modulo $p$, let $\zeta_p \in \overline{\mathbb{F}_2}$ be a primitive $p$th root of unity, and let $\left\lbrace t_1, \dots, t_{\left(p-1\right)/d}\right\rbrace \subseteq \mathbb{F}_p^{\times}$ be a transversal for the orbit of $2$.  Then, in the polynomial ring $\mathbb{F}_2\left[x\right]$, the polynomial $x^p - 1$ factors into irreducibles as $x^p - 1 = \prod_{i = 0}^{\left(p-1\right)/d} f_i$, where $f_0\left(x\right) := x - 1$ and $f_i\left(x\right) := \prod_{j = 0}^{d-1} \left(x - \zeta_p^{2^jt_i}\right)$ for all $1 \leq i \leq \frac{p-1}{d}$.
\end{lem}

\begin{proof}
This follows from the standard fact that the Frobenius map $y \mapsto y^2$ generates any finite subgroup of the Galois group $\operatorname{Gal}\left(\overline{\mathbb{F}_2}/\mathbb{F}_2\right)$ (see, e.g., \cite[\S 14.9]{Dummit}).
\end{proof}

\begin{lem} \label{C_p-stable subspaces of F_2^p}
Let $p > 2$ be prime, let $C_p$ act on $\mathbb{F}_2^p$ by shuffling coordinates, and let $d$ be the multiplicative order of $2$ modulo $p$.  Then, there are $\frac{p-1}{d} + 1$ irreducible $C_p$-stable subspaces of $\mathbb{F}_2^p$, one corresponding to each of the irreducible factors of $x^p - 1$ in the polynomial ring $\mathbb{F}_2\left[x\right]$.
\end{lem} 

\begin{proof}
Observe that $C_p$-stable subspaces of $\mathbb{F}_2^p$ correspond to ideals in the quotient ring $R := \mathbb{F}_2\left[x\right]/\left(x^p - 1\right)$.  Lemma \ref{factorization of x^p - 1 lemma} and the Chinese Remainder Theorem give the following decomposition of $R$:
\begin{equation} \label{F_2[x]/(x^p-1) decomposition}
R \cong \bigoplus_{i = 0}^{\left(p-1\right)/d} \mathbb{F}_2\left[x\right]/\left(f_i\right) \cong \mathbb{F}_2 \oplus \bigoplus_{i=1}^{\left(p-1\right)/d} \mathbb{F}_{2^d},
\end{equation}
where the $f_i$ are the irreducible polynomials defined in Lemma \ref{factorization of x^p - 1 lemma}.
\end{proof}

We thus obtain a bijective correspondence between subsets of $\left\lbrace 0, 1, \dots, \frac{p-1}{d}\right\rbrace$ and $C_p$-stable subspaces of $\mathbb{F}_2^p$.  
From here, we can explicitly construct the possibilities for $O_2\left(G\right)$.

\begin{cor} \label{representations of O_2(G)}
Let $p > 2$ be prime, let $G < \operatorname{Mon}_p\left(\mathbb{Z}\right)$ be an irreducible group, and let $d$ be the multiplicative order of $2$ modulo $p$.  Write $\mathbb{F}_2 = \left\lbrace 0, 1\right\rbrace$ and $x^p - 1 = \prod_{i=0}^{\left(p-1\right)/d}f_i$ in $\mathbb{F}_2\left[x\right]$, with the $f_i$'s as defined in Lemma \ref{factorization of x^p - 1 lemma}.  For $0 \leq i \leq \frac{p-1}{d}$, define $g_i := \prod_{j \neq i} f_j \in \mathbb{F}_2\left[x\right]$ and define the diagonal matrix $D_i \in \operatorname{D}_p\left(\mathbb{Z}\right)$ via
$$D_{i_{j,j}} := 
\begin{cases}
1 & \text{ if the }x^{j-1}\text{ coefficient of }g_i \text{ is 0} \\
-1 & \text{ if the }x^{j-1}\text{ coefficient of }g_i \text{ is 1.}
\end{cases}$$ 
Then, for some $S \subseteq \left\lbrace 0, 1, \dots, \frac{p-1}{d}\right\rbrace$, we have
$$O_2\left(G\right) = \left\langle M_{\sigma} D_i M_{\sigma}^{-1} \ | \ \sigma \in \pi\left(G\right), i \in S\right\rangle \leq \operatorname{D}_p\left(\mathbb{Z}\right),$$
where $M_{\sigma} \in \operatorname{GL}_p\left(\mathbb{Z}\right)$ is the permutation matrix corresponding to the permutation $\sigma \in \pi\left(G\right)$.
\end{cor}

\begin{proof}
By Corollary \ref{O_2(G) is a pi(G) stable subspace}, $O_2\left(G\right)$ must correspond to a $\pi\left(G\right)$-stable subspace of $\mathbb{F}_2^p$; in particular, this subspace must be $C_p$-stable, so it must preserve the irreducible summands in (\ref{F_2[x]/(x^p-1) decomposition}).  To identify a representative $g_i$ for the irreducible component $\mathbb{F}_2\left[x\right]/\left(f_i\right)$, we seek a polynomial in $\mathbb{F}_2\left[x\right]$ satisfying $g_i \equiv 0\left(\operatorname{mod }f_j\right)$ for $j \neq i$ and satisfying $g_i \not\equiv 0\left(\operatorname{mod }f_i\right)$.  The $g_i$ given in the statement satisfies these conditions.
\end{proof}

In Table \ref{reps of O_2(G) table}, we illustrate the ideas of Lemma \ref{C_p-stable subspaces of F_2^p} and Corollary \ref{representations of O_2(G)} in four key cases.  We define the all $1$'s column vector 
$$\bmatrix 1 & \cdots & 1\endbmatrix^T =: \bm{1}.$$  
By abuse of notation, we will use the same notation whether viewing $\bm{1}$ as an element of $\mathbb{F}_2^n$ or $\mathbb{Z}^n$, with the context making the meaning clear.   Further, we write 
$$V_{\bm{1}} := \left\lbrace \bm{1}, 0\right\rbrace \subseteq \mathbb{F}_2^n$$ 
and $L_{\bm{1}} \subseteq \mathbb{Z}^n$
for the $G$-lattice generated by $\bm{1} \in \mathbb{Z}^n$.

\begin{defn} \label{support length def}
The \textbf{support} of a vector $v = \bmatrix v_1 & \cdots & v_n\endbmatrix^T$ in $\mathbb{F}_2^n$ or $\mathbb{Z}^n$ is the set 
$$\operatorname{supp}\left(v\right) = \left\lbrace i \in \left\lbrace 1, 2, \dots, n\right\rbrace \ \big| \ v_i \neq 0\right\rbrace.$$  
The \textbf{support length} of $v$ is the quantity $\left|\operatorname{supp}\left(v\right)\right|$.
\end{defn}

\begin{table}
\centering
\caption{Key examples of the correspondence between subsets of $\left\lbrace 0, 1, \dots, \frac{p-1}{d}\right\rbrace$, subspaces of $\mathbb{F}_2^p$, and representations of $O_2\left(G\right) \leq \operatorname{D}_p\left(\mathbb{Z}\right)$ described in Lemma \ref{C_p-stable subspaces of F_2^p} and Corollary \ref{representations of O_2(G)}.}
 \label{reps of O_2(G) table}
\begin{tabular}{|c|c|c|}
\hline & & \\ $S \subseteq \left\lbrace 0, 1, \dots, \frac{p-1}{d}\right\rbrace$ & Subspace of $\mathbb{F}_2^p$ & $O_2\left(G\right) \leq \operatorname{D}_p\left(\mathbb{Z}\right)$ \\ & & \\\hline $\emptyset$ & $0$ & $1$ \\\hline & & \\ $\left\lbrace 0 \right\rbrace$ & $V_{\bm{1}} = \left\lbrace \bm{1}, 0\right\rbrace$ & $\left\lbrace \pm I_p\right\rbrace$ \\ & & \\\hline & & \\ $\left\lbrace 1, 2, \dots, \frac{p-1}{d}\right\rbrace$ & $V_E := \left\lbrace v \in \mathbb{F}_2^p \ \big| \ \left|\operatorname{supp}\left(v\right)\right| \text{ even}\right\rbrace$ & $\operatorname{D}_p\left(\mathbb{Z}\right) \cap \operatorname{SL}_p\left(\mathbb{Z}\right)$ \\ & & \\\hline & & \\ $\left\lbrace 0, 1, \dots, \frac{p-1}{d}\right\rbrace$ & $\mathbb{F}_2^p$ & $\operatorname{D}_p\left(\mathbb{Z}\right)$ \\ & & \\\hline
\end{tabular}
\end{table}

The following proposition gives a useful description of the subspaces in Table \ref{reps of O_2(G) table}.  We write $A_n$ for the alternating group on $n$ letters.

\begin{prop} \label{only S_p-subspaces}
For all $n \geq 3$, the only $A_n$-stable subspaces of $\mathbb{F}_2^n$ are $0$, $\mathbb{F}_2^n$, $V_{\bm{1}} = \left\lbrace \bm{1}, 0\right\rbrace$, and
$$V_E := \left\lbrace v \in \mathbb{F}_2^n \ \big| \ \left|\operatorname{supp}\left(v\right)\right| \text{ even}\right\rbrace.$$  
These four subspaces are also the only $S_n$-stable subspaces of $\mathbb{F}_2^n$ for all $n \geq 3$.
\end{prop}

\begin{proof}
We first show that the four claimed subspaces are $S_n$-stable subspaces.  It is clear that $0, V_{\bm{1}}$, and $\mathbb{F}_2^n$ are $S_n$-stable subspaces, and it is clear that the set $V_E$ is $S_n$-stable.  That $V_E$ is a subspace of $\mathbb{F}_2^n$ is a consequence of the following observation for $v, w \in \mathbb{F}_2^n$:
$$\left|\operatorname{supp}\left(v+w\right)\right| = \left|\operatorname{supp}\left(v\right)\right| + \left|\operatorname{supp}\left(w\right)\right| - 2\left|\operatorname{supp}\left(v\right) \cap \operatorname{supp}\left(w\right)\right|.$$
\par
Now, we show that these four subspaces are the only $A_n$-stable subspaces of $\mathbb{F}_2^n$; it will then immediately follow that they are the only $S_n$-stable subspaces as well.  We first observe that any $A_n$-stable subspace of $\mathbb{F}_2^n$ containing one of the standard basis vectors must be all of $\mathbb{F}_2^n$ and that any $A_n$-stable subspace containing a vector with support length $2$ must contain $V_E$. Further, since $\left[\mathbb{F}_2^n : V_E\right] = 2$, it follows that $V_E$ is a maximal proper subspace of $\mathbb{F}_2^n$, so any subspace of $\mathbb{F}_2^n$ containing $V_E$ must be $V_E$ or $\mathbb{F}_2^n$.
\par
Let $v = \bmatrix v_1 & \cdots & v_n \endbmatrix^T \in \mathbb{F}_2^n \setminus \left\lbrace 0, \bm{1}\right\rbrace$ and consider the $A_n$-subspace $\left\langle A_nv\right\rangle \subseteq \mathbb{F}_2^n$ generated by $v$.  If $\left|\operatorname{supp}\left(v\right)\right| \in \left\lbrace 1, 2\right\rbrace$, then $\left\langle A_nv\right\rangle \in \left\lbrace \mathbb{F}_2^n, V_E\right\rbrace$ by the previous paragraph and we are done.  Similarly, if $n = 3$, then every vector in $\mathbb{F}_2^n$ is either $0$, $\bm{1}$, or has support length $1$ or $2$, and we are done by the previous paragraph.  If $n \geq 4$ and $\left|\operatorname{supp}\left(v\right)\right| \geq 3$, then there exist distinct $1 \leq i, j, k, \ell \leq n$ such that $v_i = v_j = v_k = 1$ and $v_{\ell} = 0$.  Let $\sigma = \left(i, j\right)\left(k, \ell\right) \in A_n$ and observe that $\left|\operatorname{supp}\left(v + \sigma v\right)\right| = 2$.  It follows that $\left\langle A_nv\right\rangle$ contains $V_E$, so we are done.
\end{proof}

The following special class of integral vectors will be very important moving forward.

\begin{defn} \label{binary and sparse def}
We call a vector $v \in \mathbb{Z}^n$ whose entries all belong to the set $\left\lbrace 0, 1\right\rbrace$ a \textbf{binary} vector of $\mathbb{Z}^n$.
\end{defn} 

As the following result shows, binary vectors can be used to cook up nice generating sets for $G$-lattices when $G$ is a group of monomial matrices containing all diagonal matrices in $\operatorname{GL}_n\left(\mathbb{Z}\right)$.  We call a lattice $L \subseteq \mathbb{Z}^n$ \textit{primitive} if $L \neq mM$ for any lattice $M \subseteq \mathbb{Z}^n$ and integer $m \geq 2$.

\begin{lem} \label{when O_2(G) is C_2^p}
Let $n \in \mathbb{Z}^+$, let $\operatorname{D}_n\left(\mathbb{Z}\right) \leq G \leq \operatorname{Mon}_n\left(\mathbb{Z}\right)$ such that $\pi\left(G\right) \leq S_n$ contains an $n$-cycle, and let $L \subseteq \mathbb{Z}^n$ be a primitive $G$-lattice.
\begin{enumerate}
\item $L$ is generated by binary vectors in $\mathbb{Z}^n$.
\item If $L \notin \left\lbrace 0, L_{\bm{1}}\right\rbrace$ is a $G$-sublattice of $\mathbb{Z}^n$, then $L$ contains a binary vector $v \in \mathbb{Z}^n$ such that $\left|\operatorname{supp}\left(v\right)\right| \leq \left\lfloor \frac{2n}{3}\right\rfloor$, where $\left\lfloor \cdot \right\rfloor$ is the floor function.
\end{enumerate}
\end{lem}

\begin{proof}
The claim is trivially satisfied for $n = 1$, so we assume that $n \geq 2$.  To prove (a), it suffices to show that $2\mathbb{Z}^n \subseteq L$.  Since $G$ contains a permutation matrix of degree $n$, this further reduces to showing that $2e_i \in L$ for some $1 \leq i \leq n$, where the $e_i$'s are the standard basis vectors for $\mathbb{Z}^n$.  By the primitivity of $L$, there exists $v = \bmatrix v_1 & \cdots & v_n\endbmatrix^T \in L$ such that $v_j = 1$ for some $1 \leq j \leq n$.  The diagonal matrix $D$ with $D_{jj} = 1$ and $D_{ii} = -1$ for all $i \neq j$ is an element of $G$, so $v + Dv = 2e_j \in L$.
\par
Now, assume that $L \notin \left\lbrace 0, L_{\bm{1}}\right\rbrace$ and that (b) does not hold.  By (a), $L$ contains distinct binary vectors $v, w$ with support length at least $\left\lceil \frac{2n}{3}\right\rceil$.  Then, every entry of $v + w$ is a $0$, $1$, or $2$, with at least $\left\lceil \frac{n}{3}\right\rceil$ of the entries being $2$'s. The $2$'s can be eliminated by subtracting $2e_i$, all of which are in $L$ by the proof of (a), for appropriate choices of $i$, yielding a binary vector in $L$ with support length at most $\left\lfloor \frac{2n}{3}\right\rfloor$.
\end{proof}

Our next goal is to understand the primitive $G$-sublattices of $\mathbb{Z}^p$ for groups $G$ satisfying the conditions of Lemma \ref{when O_2(G) is C_2^p}.  This turns out to be extremely similar to our above discussion of the possible representations of $O_2\left(G\right)$.  

\begin{cor} \label{sublattices when O_2(G) = D_p(Z)}
Let $p > 2$ be prime, let $\operatorname{D}_p\left(\mathbb{Z}\right) \leq G \leq \operatorname{Mon}_p\left(\mathbb{Z}\right)$ such that $\pi\left(G\right) \leq S_p$ contains a $p$-cycle, and let $d$ be the multiplicative order of $2$ modulo $p$.  Write $\mathbb{F}_2 = \left\lbrace 0, 1\right\rbrace$ and $x^p - 1 = \prod_{i=0}^{\left(p-1\right)/d}f_i$ in $\mathbb{F}_2\left[x\right]$, with the $f_i$'s as defined in Lemma \ref{factorization of x^p - 1 lemma}, and let $g_0, \dots, g_n$ be as defined in Corollary \ref{representations of O_2(G)}.  For $1 \leq i \leq n$, define the binary column vector $v_i \in \mathbb{Z}^p$ by defining $v_{i_j}$ to be the $x^{j-1}$ coefficient of $g_i$ (viewing the entries of $v_i$ as elements of $\mathbb{Z}$ rather than of $\mathbb{F}_2$).  Write $L_i$ for the $G$-lattice generated by $v_i$.  Then, every primitive $G$-sublattice $L \subseteq \mathbb{Z}^p$ has the form
$$L = \sum_{i \in S} L_i$$
for some $S \subseteq \left\lbrace 0, 1, \dots, \frac{p-1}{d}\right\rbrace$.
\end{cor}

\begin{proof}
This is analogous to Corollary \ref{representations of O_2(G)}, along with Lemma \ref{when O_2(G) is C_2^p}.
\end{proof}

In Table \ref{G-sublattices table}, we explicitly give the correspondence described in Corollary \ref{sublattices when O_2(G) = D_p(Z)} in the same four cases as Table \ref{reps of O_2(G) table}.

\begin{table}
\centering
\caption[Key examples of the correspondence between subsets of $\left\lbrace 0, 1, \dots, \frac{p-1}{d}\right\rbrace$ and $G$-sublattices of $\mathbb{Z}^p$ for finite groups $\operatorname{D}_p\left(\mathbb{Z}\right) \leq G \leq \operatorname{Mon}_p\left(\mathbb{Z}\right)$ that contain a degree $p$ permutation matrix.]{Key examples of the correspondence between subsets of $\left\lbrace 0, 1, \dots, \frac{p-1}{d}\right\rbrace$ and $G$-sublattices of $\mathbb{Z}^p$ for finite groups $\operatorname{D}_p\left(\mathbb{Z}\right) \leq G \leq \operatorname{Mon}_p\left(\mathbb{Z}\right)$ that contain a $p$-cycle.  The vector $\bm{1} \in \mathbb{Z}^p$ is the all $1$'s vector, and $L_{\bm{1}}$ is the $G$-sublattice of $\mathbb{Z}^p$ generated by the $G$-orbit of $\bm{1}$.  See Table \ref{reps of O_2(G) table} or Proposition \ref{only S_p-subspaces} for the definition of $V_E$.}
\label{G-sublattices table}
\begin{tabular}{|c|c|c|}
\hline
& & \\
$S \subseteq \left\lbrace 0, 1, \dots, \frac{p-1}{d}\right\rbrace$ & Subspace of $\mathbb{F}_2^p$ & $G$-sublattice of $\mathbb{Z}^p$ \\ & & \\\hline $\emptyset$ & $0$ & $0$ \\\hline & & \\ $\left\lbrace 0 \right\rbrace$ & $V_{\bm{1}} = \left\lbrace \bm{1}, 0\right\rbrace$ & $L_{\bm{1}}$ \\ & & \\\hline & & \\ $\left\lbrace 1, 2, \dots, \frac{p-1}{d}\right\rbrace$ & $V_E$ & $L_E := \left\lbrace \left[v_i\right] \in \mathbb{Z}^p \ \big| \ 2 \ | \ \sum_{i=1}^p v_i\right\rbrace$ \\ & & \\\hline & & \\ $\left\lbrace 0, 1, \dots, \frac{p-1}{d}\right\rbrace$ & $\mathbb{F}_2^p$ & $\mathbb{Z}^p$ \\ & & \\\hline
\end{tabular}
\end{table}

\begin{cor} \label{generating with p/3 zeros}
Suppose $\operatorname{D}_p\left(\mathbb{Z}\right) \leq G \leq \operatorname{Mon}_p\left(\mathbb{Z}\right)$ satisfies the condition that $\pi\left(G\right) \leq S_p$ contains a $p$-cycle, and let $L \notin \left\lbrace 0, L_{\bm{1}}\right\rbrace$ be a primitive $G$-sublattice of $\mathbb{Z}^p$.  Then, $L$ is generated as a $G$-lattice by the $G$-orbits of at most $\frac{p-1}{d}$ binary vectors in $\mathbb{Z}^p$, each with support length at most $\left\lfloor \frac{2p}{3}\right\rfloor$.
\end{cor}

\begin{proof}
If the subset $S \subseteq \left\lbrace 0, 1, \dots, \frac{p-1}{d}\right\rbrace$ corresponding to $L$ in Corollary \ref{sublattices when O_2(G) = D_p(Z)} is $S = \left\lbrace 0, 1, \dots, \frac{p-1}{d}\right\rbrace$, then $L = \mathbb{Z}^p$, which is generated by the $G$-orbit of any of the standard basis vectors of $\mathbb{Z}^p$.  Otherwise, $\left|S\right| \leq \frac{p-1}{d}$.  For any $0 \leq i \leq \frac{p-1}{d}$, the $G$-lattice $L_i$ (as in Corollary \ref{sublattices when O_2(G) = D_p(Z)}) is generated by the $G$-orbit of any of its nonzero binary elements, as the corresponding ideal in $\mathbb{F}_2\left[x\right]/\left(x^p - 1\right)$ (see (\ref{F_2[x]/(x^p-1) decomposition})) is a field.  The support length condition then follows from (b) of Lemma \ref{when O_2(G) is C_2^p}.
\end{proof}

We can bound the orbit size of a binary vector under a subgroup of $\operatorname{Mon}_n\left(\mathbb{Z}\right)$ as a function of its support length and $\left|\pi\left(G\right)\right|$.

\begin{lem} \label{orbit size of binary vector}
Let $n \in \mathbb{Z}^+$, let $\operatorname{D}_n\left(\mathbb{Z}\right) \leq G \leq \operatorname{Mon}_n\left(\mathbb{Z}\right)$, and let $\pi:  G \to S_n$ be the projection map.  Then, for any binary vector $v \in \mathbb{Z}^n$, we have $\left|Gv\right| \leq 2^{\left|\operatorname{supp}\left(v\right)\right|}\left|\pi\left(G\right)v\right| \leq 2^{\left|\operatorname{supp}\left(v\right)\right|}\left|\pi\left(G\right)\right|$.
\end{lem}

\begin{proof}
The number of ways to permute the nonzero entries of $v$ using elements of $G$ is at most $\left|\pi\left(G\right)\right|$.  Since $\operatorname{D}_n\left(\mathbb{Z}\right) \leq G$, the number of ways to negate the nonzero entries in $v$ using elements of $G$ is $2^{\left|\operatorname{supp}\left(v\right)\right|}$.
\end{proof}



We are now ready to prove symmetric rank bounds.

\begin{prop} \label{pi(G) has only 4 subspaces}
Let $p \geq 7$ be a prime number and let $G \leq \operatorname{Mon}_p\left(\mathbb{Z}\right)$ be an irreducible group of monomial matrices satisfying $-I_p \in G$ and $\left|O_2\left(G^+\right)\right| > 1$, and let $\pi:  G \to S_p$ be the projection map.  If the only $\pi\left(G\right)$-stable subspaces of $\mathbb{F}_2^p$ are those in Proposition \ref{only S_p-subspaces}, then for any $G$-lattice $L \subseteq \mathbb{Z}^p$, we have $\operatorname{symrank}\left(L, G\right) \leq 2^p$.
\end{prop}

\begin{proof}
First, we show that $O_2\left(G\right) = \operatorname{D}_p\left(\mathbb{Z}\right)$.  By Corollary \ref{O_2(G) is a pi(G) stable subspace}, we know that $O_2\left(G\right)$ must correspond to a $\pi\left(G\right)$-stable subspace of $\mathbb{F}_2^p$.  Given our assumptions on these subspaces, the only possibilities for $O_2\left(G\right)$ are those in Table \ref{reps of O_2(G) table}.  Out of these, the only possibilities satisfying the conditions $-I_p \in G$ and $\left|O_2\left(G^+\right)\right| > 1$ is $\operatorname{D}_p\left(\mathbb{Z}\right)$.
\par
We may then apply (a) of Lemma \ref{when O_2(G) is C_2^p}, and Table \ref{G-sublattices table} shows that the only possible nontrivial $G$-sublattices of $\mathbb{Z}^p$ are $\mathbb{Z}^p, L_E$, and $L_{\bm{1}}$.  Thus, it suffices to find an element $v$ of each of these $G$-lattices whose $\operatorname{Mon}_p\left(\mathbb{Z}\right)$-orbit spans the lattice and satisfies $\left|\operatorname{Mon}_p\left(\mathbb{Z}\right)v\right| \leq 2^p$, as $\left|Gv\right| \leq \left|\operatorname{Mon}_p\left(\mathbb{Z}\right)v\right|$.  For $\mathbb{Z}^p$, we may choose $v$ to be any standard basis vector $e_i$ of $\mathbb{Z}^p$, in which case we have $\left|\operatorname{Mon}_p\left(\mathbb{Z}\right)v\right| = 2p$.  For $L_E$, we may choose any $v = e_i + e_j$ for distinct standard basis vectors $e_i, e_j \in \mathbb{Z}^p$, in which case we have $\left|\operatorname{Mon}_p\left(\mathbb{Z}\right)v\right| = \binom{p}{2} \cdot 2^2 = 2p\left(p-1\right) < 2^p$ (as $p \geq 7$).  For $L_{\bm{1}}$, we choose $v = \bm{1}$, which satisfies $\left|\operatorname{Mon}_p\left(\mathbb{Z}\right)v\right| = 2^p$.
\end{proof}

Proposition \ref{pi(G) has only 4 subspaces} now immediately gives the following result, which applies to conjecturally infinitely many primes by Artin's Conjecture on primitive roots \cite[\S 1]{HeathBrown}.

\begin{cor} \label{2 a primitive root}
Suppose $p \geq 7$ is a prime such that $2$ is a primitive root modulo $p$, and suppose $G < \operatorname{GL}_p\left(\mathbb{Z}\right)$ is an irreducible finite group satisfying (a) of Proposition \ref{Plesken prop}.  Then, for any $G$-lattice $L \subseteq \mathbb{Z}^p$, we have $\operatorname{symrank}\left(L, G\right) \leq 2^p$.
\end{cor}

\begin{proof}
Since the multiplicative order of $2$ in $\mathbb{F}_p^{\times}$ is $d = p - 1$, by Lemma \ref{C_p-stable subspaces of F_2^p} the only possible $C_p$-stable subspaces (and, thus, the only possible $\pi\left(G\right)$-stable subspaces) of $\mathbb{F}_2^p$ are the four listed in Proposition \ref{only S_p-subspaces}.
\end{proof}

For more general prime dimensions $p$, we consider the various possibilities for $\pi\left(G\right) \leq S_p$.  Since $\pi\left(G\right)$ contains a $p$-cycle by assumption, we know that $\pi\left(G\right)$ is a transitive permutation group of prime degree, which turns out to be quite restrictive.

\begin{thm} \cite[p. 177]{Burnside} \cite[\S 7.3, \S 7.4]{Cameron} \label{transitive subgroups of S_p}
If $P \leq S_p$ is a transitive permutation group of prime degree $p$, then $P$ satisfies one of the following.
\begin{enumerate}
\item $P$ is the symmetric group $S_p$.
\item $P$ is the alternating group $A_p$.
\item $P$ is a metacyclic group of order dividing $p\left(p-1\right)$.
\item There exist a prime power $q = u^t$ and $m \in \mathbb{Z}^+$, with $\left(m, q\right) \notin \left\lbrace \left(2, 2\right), \left(2, 3\right)\right\rbrace$, such that $p = \frac{q^m-1}{q-1}$ and the socle of $P$ is isomorphic to the $m$-dimensional projective special linear group $\operatorname{PSL}_m\left(q\right)$ over the finite field with $q$ elements.  In this case, 
$$\left|P\right| \leq \left|\operatorname{PSL}_m\left(q\right)\right|\gcd\left(m, q-1\right)t.$$
\end{enumerate}
\end{thm}

Because of (d) in Theorem \ref{transitive subgroups of S_p}, the following formula will be useful.

\begin{prop} \cite[\S 3.3.1]{Wilson} \label{order of PSL(m,q)}
For all $m \geq 2$ and prime powers $q$, we have
$$\left|\operatorname{PSL}_m\left(q\right)\right| = \frac{q^{m\left(m-1\right)/2}}{\gcd\left(m,q-1\right)}\prod_{i=2}^m \left(q^i - 1\right).$$
\end{prop}

We now prove an asymptotic result about symmetric ranks in prime dimensions.

\begin{thm} \label{p-1/a theorem}
For all $a \in \mathbb{Z}^+$, there exists $N_a \in \mathbb{Z}^+$ such that, for all primes $p \geq N_a$ for which the multiplicative order of $2$ modulo $p$ is equal to $\frac{p-1}{a}$ and all irreducible finite subgroups $G < \operatorname{GL}_p\left(\mathbb{Z}\right)$ satisfying (a) of Proposition \ref{Plesken prop}, we have $\operatorname{symrank}\left(\mathbb{Z}^p, G\right) \leq 2^p$.
\end{thm}

\begin{proof}
Fix $a \in \mathbb{Z}^+$, and suppose that $p \in \mathbb{Z}^+$ is a prime such that the multiplicative order of $2$ modulo $p$ is equal to $d = \frac{p-1}{a}$.  By (a) of Proposition \ref{Plesken prop} and Proposition \ref{reduction to monomial matrices}, it suffices to consider $G \leq \operatorname{Mon}_p\left(\mathbb{Z}\right)$ satisfying $-I_p \in G$ and $\left|O_2^+\left(G\right)\right| > 1$, and primitive $G$-sublattices $L \subseteq \mathbb{Z}^p$.  By (c) of Observation \ref{when -I_n in G} and (a) of Proposition \ref{Plesken prop}, we have $O_2\left(G\right) \cong C_2^{d\ell + 1} \cong C_2^{\ell\left(p-1\right)/a + 1}$ for some integer $1 \leq \ell \leq a$.  By Lemma \ref{O_2(G) consists of diagonal matrices}, we have $G \cong O_2\left(G\right) \rtimes \pi\left(G\right)$, where $\pi:  G \to S_p$ is the projection map.  Since $\pi\left(G\right)$ contains a $p$-cycle by (a) of Proposition \ref{Plesken prop}, we may proceed by considering the possibilities for $\pi\left(G\right)$ given by Theorem \ref{transitive subgroups of S_p}.
\par
\textit{Case I}:  If $\pi\left(G\right) \in \left\lbrace A_p, S_p\right\rbrace$, then we are done by Propositions \ref{only S_p-subspaces} and \ref{pi(G) has only 4 subspaces}.
\par
\textit{Case II}:  Assume that $\pi\left(G\right)$ is a metacyclic group of order dividing $p\left(p-1\right)$.  We then have $\left|\pi\left(G\right)\right| \leq p\left(p-1\right)$, so 
$$\left|G\right| \leq 2^{\ell\left(p-1\right)/a + 1}p\left(p-1\right).$$  
\par
\textit{Subcase i}:  Suppose $1 \leq \ell \leq a - 1$.  Then, (a) of Lemma \ref{symmetric rank facts} gives the bound
$$\operatorname{symrank}\left(L, G\right) \leq \left|G\right|p \leq 2^{\ell\left(p-1\right)/a + 1}p^2\left(p-1\right) \leq 2^{\left(a-1\right)\left(p-1\right)/a + 1}p^2\left(p-1\right).$$
Thus, it suffices to show that $2^p$ is greater than or equal to the rightmost expression.  This is equivalent to verifying the inequality
$$2^{\left(p-1\right)/a} \geq p^2\left(p-1\right).$$
Since $a$ is fixed, this inequality is satisfied for sufficiently large $p$.
\par
\textit{Subcase ii}:  Suppose $\ell = a$.  Then, we have $O_2\left(G\right) \cong C_2^p$, so $D_p\left(\mathbb{Z}\right) \leq G$.  Thus, we apply (a) of Lemma \ref{symmetric rank facts}, Corollary \ref{generating with p/3 zeros}, and Lemma \ref{orbit size of binary vector} to attain the bound
$$\operatorname{symrank}\left(L, G\right) \leq a \cdot 2^{\left\lfloor 2p/3\right\rfloor} p\left(p-1\right).$$
Similar to subcase (i), it suffices to verify that 
$$2^{p/3} \geq ap\left(p-1\right),$$
which holds for sufficiently large $p$ since $a$ is fixed.
\par
\textit{Case III}:  Assume that $\pi\left(G\right)$ satisfies (d) of Theorem \ref{transitive subgroups of S_p}; in particular, suppose the socle of $G$ is isomorphic to $\operatorname{PSL}_m\left(q\right)$, where $q = u^t$ for some prime $u$ and $t \in \mathbb{Z}^+$.  Note that we have $p = \frac{q^m-1}{q-1}$ in this case.  Observe that $2 \leq q \leq p$ and that $m = \log_q\left[p\left(q-1\right)+1\right]$.  We proceed by considering the same two subcases as in case II.
\par
\textit{Subcase i}:  Suppose $1 \leq \ell \leq a - 1$.  In this case, we have
$$\left|G\right| \leq 2^{\ell\left(p-1\right)/a + 1} \cdot \frac{q^{m\left(m-1\right)/2}}{\gcd\left(m,q-1\right)} \cdot \prod_{i=2}^m\left(q^i - 1\right) \cdot t \cdot \gcd\left(m, q-1\right).$$
Using some algebra, (a) of Lemma \ref{symmetric rank facts}, and the fact that $\ell \leq a - 1$, we see that it suffices to verify the following inequality:
\begin{equation} \label{l < a equivalent stupid bound}
2^{\frac{p-1}{a}} \geq q^{m\left(m-1\right)/2} \cdot \prod_{i=2}^m \left(q^i-1\right) \cdot t \cdot p.
\end{equation}
To this end, observe that
$$\prod_{i=2}^m \left(q^i-1\right) < \prod_{i=2}^m q^i = q^{\sum_{i=2}^m i} = q^{m\left(m+1\right)/2 - 1}.$$
In particular,
$$q^{m\left(m-1\right)/2}q^{m\left(m+1\right)/2 - 1} = q^{m^2 - 1}.$$
Since $2 \leq q \leq p$, we have
$$m = \log_q\left[p\left(q-1\right)+1\right] \leq \log_q\left(p^2+1\right) \leq \log_2\left(p^2+1\right).$$
Therefore,
$$q^{m^2-1} < q^{m^2} \leq q^{\log_q\left(p^2+1\right) \cdot \log_2\left(p^2+1\right)} = \left(p^2+1\right)^{\log_2\left(p^2+1\right)}.$$
Further, since the prime $u$ must be at least $2$, we have
$$t = \log_u\left(q\right) \leq \log_2\left(q\right) \leq \log_2\left(p\right).$$
Thus, we obtain the following inequality which is stronger than inequality (\ref{l < a equivalent stupid bound}):
$$2^{\frac{p-1}{a}} \geq \left(p^2+1\right)^{\log_2\left(p^2+1\right)} \cdot \log_2\left(p\right) \cdot p.$$
Since $a$ is fixed, this inequality holds for sufficiently large $p$, as desired.
\par
\textit{Subcase ii:}  Suppose $\ell = a$.  In this case, we have $O_2\left(G\right) = \operatorname{D}_p\left(\mathbb{Z}\right)$.  Thus, we apply Corollary \ref{generating with p/3 zeros} and Lemma \ref{orbit size of binary vector} to attain the following bound:
\begin{align*}
\operatorname{symrank}\left(L, G\right) & \leq \frac{p-1}{d} \cdot 2^{\left\lfloor  2p/3\right\rfloor} \cdot \left|\pi\left(G\right)\right| \\
& \leq a \cdot 2^{\left\lfloor  2p/3\right\rfloor} \cdot \frac{q^{m\left(m-1\right)/2}}{\gcd\left(m,q-1\right)} \cdot \prod_{i=2}^m\left(q^i-1\right) \cdot \gcd\left(m,q-1\right) \cdot t \\
& = a \cdot 2^{\left\lfloor  2p/3\right\rfloor} \cdot q^{m\left(m-1\right)/2} \prod_{i=2}^m\left(q^i-1\right) \cdot t.
\end{align*}
Applying the bounds derived in subcase (i) and removing the floor from the exponent of the $2$, it suffices to verify the following inequality:
$$2^{p/3} \geq a \cdot \left(p^2+1\right)^{\log_2\left(p^2+1\right)} \cdot \log_2\left(p\right).$$
Since $a$ is fixed, this inequality holds for all sufficiently large $p$, completing the proof.
\end{proof}

In particular, Corollary \ref{2 a primitive root} asserts that when $a = 1$, we can take $N_1 = 2$; in other words, $\operatorname{symrank}_{\operatorname{irr}}\left(p\right) \leq 2^p$ for \textit{all} primes $p$ for which the multiplicative order of $2$ modulo $p$ is $p - 1$.  We now show the same result in the $a = 2$ case.

\begin{thm} \label{p-1/2 theorem}
If $p \geq 7$ is a prime number such that the order of $2$ modulo $p$ is $\frac{p-1}{2}$, and $G < \operatorname{GL}_p\left(\mathbb{Z}\right)$ is an irreducible finite subgroup satisfying (a) of Proposition \ref{Plesken prop}, then $\operatorname{symrank}\left(\mathbb{Z}^p, G\right) \leq 2^p$.
\end{thm}

\begin{proof}
We follow the proof of Theorem \ref{p-1/a theorem}, with $a = 2$.  In particular, we must have $\ell \in \left\lbrace 1, 2\right\rbrace$.  Note that the case $p \leq 23$ is addressed in Chapter \ref{small dimensions chapter}, so it suffices to prove the claim for primes $p \geq 41$.  The justification given for case I of Theorem \ref{p-1/a theorem} applies here as well, so we proceed to case II.
\par 
\textit{Case II:}  Assume that $\pi\left(G\right)$ is a metacyclic group of order dividing $p\left(p-1\right)$.  As above, we have
$$\left|G\right| \leq 2^{\ell\left(p-1\right)/2 + 1}p\left(p-1\right).$$  
\par
\textit{Subcase i}:  If $\ell = 1$, then (a) of Lemma \ref{symmetric rank facts} gives the bound
$$\operatorname{symrank}\left(L, G\right) \leq \left|G\right|p = 2^{\left(p-1\right)/2 + 1}p^2\left(p-1\right),$$
and the right hand side is less than or equal to $2^p$ for $p \geq 31$ (verified using a computer algebra system).  
\par
\textit{Subcase ii}:  If $\ell = 2$, then $D_p\left(\mathbb{Z}\right) \leq G$, so we apply (a) of Lemma \ref{symmetric rank facts}, Corollary \ref{generating with p/3 zeros}, and Lemma \ref{orbit size of binary vector} to attain the bound
$$\operatorname{symrank}\left(L, G\right) \leq 2 \cdot 2^{\left\lfloor 2p/3\right\rfloor} p\left(p-1\right).$$
The right hand side is less than or equal to $2^p$ for all $p \geq 31$ (verified using a computer algebra system).
\par
\textit{Case III}:  Assume that $\pi\left(G\right)$ satisfies (d) of Theorem \ref{transitive subgroups of S_p}; in particular, suppose the socle of $G$ is isomorphic to $\operatorname{PSL}_m\left(q\right)$, where $q = u^t$ for some prime $u$ and $t \in \mathbb{Z}^+$.  Note that we have $p = \frac{q^m-1}{q-1}$ in this case.  By \cite[Table II]{Bateman}, the first prime $p$ of this form that is at least $41$ and satisfies the condition that $d = \frac{p-1}{2}$ is the prime $p = 2801$, obtained by $q = 7$ and $m = 5$.  Following the ideas in case III of the proof of Theorem \ref{p-1/a theorem}, we find that $2^p \geq \operatorname{symrank}\left(L, G\right)$ for $p \geq 760$ when $\ell = 1$ and for $p \geq 1297$ when $\ell = 2$.  As $760 < 1297 < 2801$, this completes the proof.
\end{proof}




\section{Almost Simple Groups in Prime Dimensions} \label{almost simple groups chapter}

Let $p \geq 3$ be a prime.  Recall that (b) of Proposition \ref{Plesken prop} says the following about an irreducible finite matrix group $G < \operatorname{GL}_p\left(\mathbb{Z}\right)$:  The group $O_2\left(G^+\right) = 1$; the group $G^+$ has a unique minimal normal subgroup $N \neq 1$ (possibly $N = G^+$); the group $N$ is nonabelian simple; the centralizer $C_{G^+}\left(N\right) = 1$; and $N$ is irreducible as a matrix group.
\par
In this section, we prove the following theorem.

\begin{thm} \label{almost simple groups theorem}
Let $p \geq 7$ be a prime number and let $G < \operatorname{GL}_p\left(\mathbb{Z}\right)$ be an irreducible finite group satisfying (b) of Proposition \ref{Plesken prop}.  Then, for any $G$-lattice $L \subseteq \mathbb{Z}^p$, we have $\operatorname{symrank}\left(L, G\right) \leq 2^p$.
\end{thm}

Theorem \ref{low dimensions theorem} implies Theorem \ref{almost simple groups theorem} for $p \in \left\lbrace 7, 11, 13, 17, 19, 23\right\rbrace$, so it suffices to consider $p \geq 29$ in this section.  Further, observe that, together with Theorem \ref{low dimensions theorem}, Corollary \ref{2 a primitive root}, and Theorems \ref{p-1/a theorem} and \ref{p-1/2 theorem}, Theorem \ref{almost simple groups theorem} completes the proof of Theorems \ref{asymptotic theorem} and \ref{symrankirr theorem 1}.  In fact, Theorem \ref{almost simple groups theorem} gives us something stronger, as this theorem does not have any restrictions on the order of $2$ modulo $p$.
\par
Throughout this section, we let $G < \operatorname{GL}_p\left(\mathbb{Z}\right)$ be an irreducible finite group for some prime $p \geq 29$.  We assume that $G^+$ satisfies (b) of Proposition \ref{Plesken prop}; in particular, we have $S \trianglelefteq G^+ \leq \operatorname{Aut}\left(S\right)$ for some nonabelian finite simple group $S$, so $G^+$ is an \textit{almost simple group}.  We will follow notation from \cite{ATLAS} for the finite simple groups.  For an abstract group $S$, we write $\operatorname{rdim}\left(S\right)$ for the minimal degree of a faithful representation of $S$.  It will be useful to have bounds on $\operatorname{rdim}\left(S\right)$ when $S$ is a finite simple group.  
\par
In many cases, bounds on the degrees of \textit{projective representations} of finite simple groups are more accessible.  To that end, recall that a \textit{projective representation} of a group $G$ is an injective group homomorphism $G \hookrightarrow \operatorname{PGL}_n\left(k\right)$ for some $n \in \mathbb{Z}^+$ and some field $k$, where $\operatorname{PGL}_n\left(k\right)$ is the $n$-dimensional \textit{projective general linear group} over $k$.  We now define the following analogue of representation dimension.

\begin{defn} \label{prdim def}
We define the \textbf{projective representation dimension} of a group $G$, denoted $\operatorname{prdim}\left(G\right)$, to be the minimal $n \in \mathbb{Z}^+$ for which $G$ admits a projective representation $G \hookrightarrow \operatorname{PGL}_n\left(k\right)$ for some field $k$.
\end{defn}

\begin{lem} \label{prdim is lower bound on rdim}
For all finite simple groups $S$, we have $\operatorname{prdim}\left(S\right) \leq \operatorname{rdim}\left(S\right)$.
\end{lem}

\begin{proof}
Suppose $\rho:  S \hookrightarrow \operatorname{GL}_n\left(k\right)$ is a representation for some $n \in \mathbb{Z}^+$ and some field $k$, and consider the quotient map $\phi:  \operatorname{GL}_n\left(k\right) \to \operatorname{PGL}_n\left(k\right)$.  Since $S$ is simple, it follows that the restriction of $\phi$ to $\rho\left(S\right)$ is injective.  Thus, the map $\phi \circ \rho$ is a faithful projective representation of $S$ of the same degree as $\rho$.
\end{proof}

Our goal is now to apply (a) of Lemma \ref{symmetric rank facts} to develop a bound on $\operatorname{symrank}\left(L, G\right)$ that is dependent only upon $S$.  We start with a general numerical observation.

\begin{lem} \label{numerical observation}
For all $a, b, c \in \mathbb{Z}^+$, if $b \geq a$ and $2^a \geq ac$, then $2^b \geq bc$.
\end{lem}

\begin{proof}
Since $2^a \geq ac \geq c$ and $2^{b-a} - 1 \geq b-a$, we have
$$2^b - 2^a = 2^a\left(2^{b-a}-1\right) \geq c\left(b-a\right) = bc - ac,$$
and the claim follows.
\end{proof}

\begin{lem} \label{almost simple groups inequality}
Suppose $S \trianglelefteq G^+ \leq \operatorname{Aut}\left(S\right)$ for some nonabelian finite simple group $S$ and irreducible finite matrix group $G < \operatorname{GL}_p\left(\mathbb{Z}\right)$ for $p \geq 29$ prime.  If $S$ satisfies any of the following inequalities, then $\operatorname{symrank}\left(L, G\right) \leq 2^p$ for all $G$-lattices $L \subseteq \mathbb{Z}^p$.
\begin{equation} \label{og almost simple inequality}
2^{\operatorname{rdim}\left(S\right)} \geq 2\left|\operatorname{Aut}\left(S\right)\right|\operatorname{rdim}\left(S\right) 
\end{equation}
\begin{equation} \label{projective almost simple inequality}
2^{\operatorname{prdim}\left(S\right)} \geq 2\left|\operatorname{Aut}\left(S\right)\right|\operatorname{prdim}\left(S\right)
\end{equation}
\begin{equation} \label{29 almost simple inequality}
2^{29} \geq 58\left|\operatorname{Aut}\left(S\right)\right|
\end{equation}
\end{lem}

\begin{proof}
The first inequality follows from (a) of Lemma \ref{symmetric rank facts}, Lemma \ref{numerical observation}, and the observations that $\operatorname{rdim}\left(S\right) \leq \operatorname{rdim}\left(G\right)$ and 
$$\left|G\right| = 2\left|G^+\right| \leq 2\left|\operatorname{Aut}\left(S\right)\right|.$$  
The second inequality follows from the first inequality and Lemmas \ref{numerical observation} and \ref{prdim is lower bound on rdim}.  The third inequality follows from the first inequality, Lemma \ref{numerical observation}, and Theorem \ref{low dimensions theorem}.
\end{proof}

We now prove Theorem \ref{almost simple groups theorem}.

\begin{proof} [Proof of Theorem \ref{almost simple groups theorem}]
We consider three cases:  when $S$ is the alternating group $A_n$ for some $n \geq 5$, when $S$ is a member of another infinite family of finite simple groups, and when $S$ is one of the sporadic finite simple groups (although the methodologies in the second and third cases are essentially the same).
\par
\textit{Case I:  }Suppose that $S \cong A_n$ for some $n \geq 5$.  By \cite[Thm. 2.5.7]{James}, all irreducible rational representations of $A_n$ are restrictions of irreducible rational representations of the full symmetric group $S_n$.  Thus, by (c) of Lemma \ref{symmetric rank facts}, it suffices to consider the case $G^+ \cong S_n$.  By Theorem 2.3.21 of \cite{James}, the degree of an irreducible representation of $S_n$ divides $n!$, and Theorem 3.4.10.ii of the same text states that the minimal degree of a faithful irreducible representation of $S_n$ is equal to $n-1$.  Furthermore, Theorem 3.4.10.iii of \cite{James} says that for $n \geq 6$, there does not exist a faithful irreducible representation of $S_n$ of degree $n$.  It now follows that the only possible prime degree of an irreducible representation of $S_n$ is $n-1$.  
\par
The Weyl group $W = W\left(\mathsf{A}_{n-1}\right) \cong S_n$, so all $G^+$-lattices (and, thus, all $G$-lattices) of rank $n-1$ are $W$-lattices of the same rank.  Therefore, since $\left|G\right| = 2\left|G^+\right|$, the desired bound now follows from Theorem \ref{root systems theorem}.
\par
\textit{Case II:  }Suppose $S$ is a member of one of the infinite families of finite simple groups aside from the alternating groups.  We first apply Lemma \ref{almost simple groups inequality} to each of the infinite families of finite simple groups to verify Theorem \ref{almost simple groups theorem} for all but finitely many cases; see Table \ref{almost simple groups table} for details.  In each case, the strategy is to verify one of the inequalities in Lemma \ref{almost simple groups inequality} for some value(s) of $m$ and/or $q$, and then prove that as $m$ and/or $q$ increase the appropriate inequality still holds.  For the remaining cases, we use \cite{ATLAS} and a computer algebra system to find the minimal possible prime $p$ such that a faithful representation of a group $G$ satisfying $S \trianglelefteq G \leq \operatorname{Aut}\left(S\right)$ exists.  By Lemma \ref{numerical observation}, it suffices to verify the inequality $2^p \geq 2\left|\operatorname{Aut}\left(S\right)\right|p$ in these cases.  This is shown in Table \ref{almost simple exceptions table}.  After using this approach, there are no more cases to consider.
\par
\textit{Case III:  }Assume that $S$ is one of the 26 sporadic finite simple groups.  We begin by applying Lemma \ref{almost simple groups inequality} to each group, getting group information from \cite{ATLAS}.  This strategy proves the theorem for all but the following six groups:  $\operatorname{M}_{23}$, $\operatorname{M}_{24}$, $\operatorname{Co}_3$, $\operatorname{Co}_2$, $\operatorname{HS}$, and $\operatorname{M}^{\operatorname{c}}\operatorname{L}$.  For none of these six groups does there exist a potential $S \trianglelefteq G \leq \operatorname{Aut}\left(S\right)$ such that $G$ admits an irreducible representation in prime degree $p \geq 29$, so we are done.
\end{proof}

\begin{table}[h]
\centering
\caption[Reducing the proof of Theorem \ref{almost simple groups theorem} to finitely many cases using bounds on the (projective) representation dimensions of finite simple groups.]{Information used in the proof of Theorem \ref{almost simple groups theorem} to check the inequalities in Lemma \ref{almost simple groups inequality} when $S$ is a member of one of the infinite families of finite simple groups besides the alternating groups.  We follow notation from \cite{ATLAS} for our groups in the first column.  In the second column, a ``$p$'' indicates that a lower bound on $\operatorname{prdim}\left(S\right)$ was used and an ``$r$'' indicates that a lower bound on $\operatorname{rdim}\left(S\right)$ was used; we also provide a citation for the bound.  In the third column, we list all values of $q$, $n$, or ordered pairs $\left(n,q\right)$ for which none of the inequalities in Lemma \ref{almost simple groups inequality} are satisfied.}
\label{almost simple groups table}
\begin{tabular}{c|c|c} Finite Simple Group $S$ & $\operatorname{rdim}\left(S\right)$ or $\operatorname{prdim}\left(S\right)$ Bound & Remaining Cases \\\hline $L_2\left(q\right)$, $q \equiv 1\left(\operatorname{mod }4\right)$ & $r, \frac{q+1}{2}$, \cite{Adams} & None \\\hline $L_2\left(q\right)$, $q \equiv 3\left(\operatorname{mod }4\right)$ & $r$, $\frac{q-1}{2}$, \cite{Adams} & None \\\hline $L_2\left(q\right)$, $q \geq 2$ even & $r$, $q - 1$, \cite{Adams} & None \\\hline $L_n\left(q\right)$, $n \geq 3$ & $p$, $\frac{q^n - 1}{q-1} - n$, \cite{Seitz} & $\left(5, 2\right)$ \\\hline $O_{2n+1}\left(3\right)$, $n \geq 1$ & $p$, $\frac{3^{2n} - 1}{8} - \frac{3^n - 1}{2}$, \cite{Seitz} & None \\\hline $O_{2n+1}\left(q\right)$, $n \geq 1$, $q \geq 5$ odd & $p$, $\frac{q^{2n}-1}{q^2-1} - n$, \cite{Seitz} & None \\\hline $S_4\left(q\right)$, $q \geq 2$ even & $p$, $\frac{q\left(q-1\right)^2}{2}$, \cite{Seitz} & None \\\hline $S_4\left(q\right)$, $q \geq 3$ odd & $p$, $\frac{1}{2}\left(q^2-1\right)$, \cite{Seitz} & $5$, $7$, $9$ \\\hline $S_{2n}\left(q\right)$, $n \geq 3$, $q \geq 2$ even & $p$, $\frac{q\left(q^n-1\right)\left(q^{n-1}-1\right)}{2\left(q+1\right)}$, \cite{Seitz} & $\left(4, 2\right)$ \\\hline $S_{2n}\left(q\right)$, $n \geq 3$, $q \geq 3$ odd & $p$, $\frac{1}{2}\left(q^n - 1\right)$, \cite{Seitz} & $\left(3, 3\right)$, $\left(4, 3\right)$ \\\hline $O_8^+\left(q\right)$, $q \notin \left\lbrace 2, 3, 5\right\rbrace$ & $p$, $\left(q^3-1\right)\left(q^2+1\right)$, \cite{Landazuri} & None \\\hline $O_8^+\left(q\right)$, $q \in \left\lbrace 2, 3, 5\right\rbrace$ & $p$, $q^2\left(q^3 - 1\right)$, \cite{Landazuri} & $2$ \\\hline $O_{2n}^+\left(q\right)$, $n > 4$ even, $q \notin \left\lbrace 2, 3, 5\right\rbrace$ & $p$, $\left(q^{n-1}-1\right)\left(q^{n-2}+1\right)$, \cite{Landazuri} & None \\\hline $O_{2n}^+\left(q\right)$, $n > 4$ even, $q \in \left\lbrace 2, 3, 5\right\rbrace$ & $p$, $q^{n-2}\left(q^{n-1}-1\right)$, \cite{Landazuri} & None \\\hline $O_{2n}^+\left(q\right)$, $n > 4$ odd, $q \notin \left\lbrace 2, 3, 5\right\rbrace$ & $p$, $\left(q^{n-1} - 1\right)\left(q^{n-2}+1\right)$, \cite{Landazuri} & None \\\hline $O_{2n}^+\left(q\right)$, $n > 4$ odd, $q \in \left\lbrace 2, 3, 5\right\rbrace$ & $p$, $q^{n-2}\left(q^{n-1}+1\right)$, \cite{Landazuri} & None \\\hline $E_6\left(q\right)$ & $p$, $q^9\left(q^2-1\right)$, \cite{Seitz} & None \\\hline $E_7\left(q\right)$ & $p$, $q^{15}\left(q^2-1\right)$, \cite{Seitz} & None \\\hline $E_8\left(q\right)$ & $p$, $q^{27}\left(q^2-1\right)$, \cite{Seitz} & None \\\hline $F_4\left(q\right)$, $q$ odd & $p$, $q^6\left(q^2-1\right)$, \cite{Seitz} & None \\\hline $F_4\left(q\right)$, $q$ even & $p$, $\frac{1}{2}q^7\left(q^3-1\right)\left(q-1\right)$, \cite{Seitz} & None \\\hline $G_2\left(q\right)$, $q > 2$, $3 \not | \ q$ & $p$, $q\left(q^2-1\right)$, \cite{Seitz} & None \\\hline $G_2\left(q\right)$, $3 \ | \ q$ & $p$, $q\left(q^2 - 1\right)$, \cite{Seitz} & None \\\hline $U_{n+1}\left(q\right)$, $n \geq 2$ even & $p$, $\frac{q\left(q^n-1\right)}{q+1}$, \cite{Seitz} & $\left(4, 2\right)$, $\left(6, 2\right)$ \\\hline $U_{n+1}\left(q\right)$, $n \geq 3$ odd & $p$, $\frac{q^{n+1}-1}{q+1}$, \cite{Seitz} & $\left(3, 3\right)$, $\left(5, 2\right)$ \\\hline $^2E_6\left(q\right)$ & $p$, $q^9\left(q^2-1\right)$, \cite{Seitz} & None \\\hline $^3D_4\left(q\right)$ & $p$, $q^3\left(q^2-1\right)$, \cite{Seitz} & $2$ \\\hline $Sz\left(2^{2n+1}\right)$ & $r$, $q^2$, \cite{Suzuki} & None \\\hline $^2F_4\left(2^{2n+1}\right)$, $n \geq 1$ & $p$, $\left(\frac{q}{2}\right)^{1/2}q^4\left(q-1\right)$, \cite{Seitz} & None \\\hline $^2G_2\left(3^{2n+1}\right)$, $n \geq 1$ & $r$, $q^2 - q + 1$, \cite{Ward} & None \\\hline
\end{tabular}
\end{table}


\begin{table}[h]
\centering
\caption[Minimal prime dimensions $\geq 29$ such that there exists a degree $p$ irreducible faithful representation of a group $G$ satisfying $S \trianglelefteq G \leq \operatorname{Aut}\left(S\right)$ for certain finite simple groups $S$.]{For the remaining cases in Table \ref{almost simple groups table}, we find the minimal prime $p \geq 29$ such that there exists a degree $p$ irreducible faithful representation of a group $G$ satisfying $S \trianglelefteq G \leq \operatorname{Aut}\left(S\right)$.  These were computed using character tables in \cite{GAP}.  In many cases, no such $p$ exists.
\\
*There are two irreducible representations of this degree, but they are not rational.  There are no other irreducible representations in prime degrees.}
\label{almost simple exceptions table}
\begin{tabular}{c|c|c|c}
$S$ & $\left|\operatorname{Aut}\left(S\right)\right|$ & $\operatorname{Out}\left(S\right)$ & Min. Prime Dim. $\geq 29$ \\\hline $L_5\left(2\right)$ & $19998720$ & $C_2$ & None \\\hline $S_4\left(5\right)$ & $9360000$ & $C_2$ & None \\\hline $S_4\left(7\right)$ & $276595200$ & $C_2$ & None \\\hline $S_4\left(9\right)$ & $6886425600$ & $C_2^2$ & $41$ \\\hline $S_8\left(2\right)$ & $47377612800$ & $C_1$ & None \\\hline $S_6\left(3\right)$ & $9170703360$ & $C_2$ & None \\\hline $S_8\left(3\right)$ & $131569513308979200$ & $C_2$ & $41$* \\\hline $O_8^+\left(2\right)$ & $1045094400$ & $S_3$ & None \\\hline $U_5\left(2\right)$ & $27371520$ & $C_2$ & None \\\hline $U_7\left(2\right)$ & $455574206545920
$ & $C_2$ & $43$* \\\hline $U_4\left(3\right)$ & $26127360$ & $D_4$ & None \\\hline $U_6\left(2\right)$ & $55180984320$ & $S_3$ & None \\\hline $^3D_4\left(2\right)$ & $634023936$ & $C_3$ & None
\end{tabular}
\end{table}

\clearpage
\bibliographystyle{alpha}
\bibliography{DissertationReferences}

\newcommand{\etalchar}[1]{$^{#1}$}
\begin{thebibliography}{LMMR13}

\bibitem[Ada02]{Adams}
Jeffrey Adams.
\newblock Character tables for $\operatorname{GL}\left(2\right)$,
  $\operatorname{SL}\left(2\right)$, $\operatorname{PGL}\left(2\right)$ and
  $\operatorname{PSL}\left(2\right)$ over a finite field.
\newblock {\em Lecture Notes, University of Maryland}, 25:26--28, 2002.

\bibitem[BMKS16]{Bardestani}
Mohammad Bardestani, Keivan Mallahi-Karai, and Hadi Salmasian.
\newblock Minimal dimension of faithful representations for {$p$}-groups.
\newblock {\em J. Group Theory}, 19(4):589--608, 2016.

\bibitem[BS62]{Bateman}
Paul~T. Bateman and Rosemarie~M. Stemmler.
\newblock Waring's problem for algebraic number fields and primes of the form
  {$(p\sp{r}-1)/(p\sp{d}-1)$}.
\newblock {\em Illinois J. Math.}, 6:142--156, 1962.

\bibitem[Bur00]{Burnside}
W~Burnside.
\newblock On some properties of groups of odd order.
\newblock {\em Proceedings of the London Mathematical Society}, 1(1):162--184,
  1900.

\bibitem[Cam99]{Cameron}
Peter~J Cameron.
\newblock {\em Permutation groups}.
\newblock Number~45. Cambridge University Press, 1999.

\bibitem[Car72]{Carter}
R.~W. Carter.
\newblock Conjugacy classes in the {W}eyl group.
\newblock {\em Compositio Math.}, 25:1--59, 1972.

\bibitem[CCN{\etalchar{+}}85]{ATLAS}
J.~H. Conway, R.~T. Curtis, S.~P. Norton, R.~A. Parker, and R.~A. Wilson.
\newblock {\em {$\Bbb{ATLAS}$} of finite groups}.
\newblock Oxford University Press, Eynsham, 1985.
\newblock Maximal subgroups and ordinary characters for simple groups, With
  computational assistance from J. G. Thackray.

\bibitem[CKR11]{Cernele}
Shane Cernele, Masoud Kamgarpour, and Zinovy Reichstein.
\newblock Maximal representation dimension of finite {$p$}-groups.
\newblock {\em J. Group Theory}, 14(4):637--647, 2011.

\bibitem[Cov00]{Covello}
Sandra Covello.
\newblock {\em Minimal parabolic subgroups in the symmetric groups}.
\newblock PhD thesis, University of Birmingham, 2000.

\bibitem[Dad65]{Dade}
E.~C. Dade.
\newblock The maximal finite groups of {$4\times 4$} integral matrices.
\newblock {\em Illinois J. Math.}, 9:99--122, 1965.

\bibitem[DF04]{Dummit}
David~S. Dummit and Richard~M. Foote.
\newblock {\em Abstract algebra}.
\newblock Wiley, New York, 3rd ed edition, 2004.

\bibitem[GAP22]{GAP}
The GAP~Group.
\newblock {\em {GAP -- Groups, Algorithms, and Programming, Version 4.12.2}},
  2022.

\bibitem[HB86]{HeathBrown}
DR~Heath-Brown.
\newblock Artin's conjecture for primitive roots.
\newblock {\em The Quarterly Journal of Mathematics}, 37(1):27--38, 1986.

\bibitem[Hea24]{dissertation}
Jason~Bailey Heath.
\newblock {\em Representation Dimensions of Algebraic Tori and Symmetric Ranks
  of $G$-Lattices}.
\newblock PhD thesis, University of South Carolina, Columbia, SC, August 2024.

\bibitem[Hum72]{Humphreys}
James~E. Humphreys.
\newblock {\em Introduction to {L}ie algebras and representation theory},
  volume Vol. 9 of {\em Graduate Texts in Mathematics}.
\newblock Springer-Verlag, New York-Berlin, 1972.

\bibitem[JK84]{James}
Gordon James and Adalbert Kerber.
\newblock {\em The Representation Theory of the Symmetric Group}.
\newblock Encyclopedia of Mathematics and its Applications. Cambridge
  University Press, 1984.

\bibitem[JLY02]{Jensen}
Christian~U Jensen, Arne Ledet, and Noriko Yui.
\newblock {\em Generic polynomials: constructive aspects of the inverse Galois
  problem}, volume~45.
\newblock Cambridge University Press, 2002.

\bibitem[KM08]{Karpenko}
Nikita~A. Karpenko and Alexander~S. Merkurjev.
\newblock Essential dimension of finite {$p$}-groups.
\newblock {\em Invent. Math.}, 172(3):491--508, 2008.

\bibitem[Lem04]{Lemire}
Nicole Lemire.
\newblock Essential dimension of algebraic groups and integral representations
  of {W}eyl groups.
\newblock {\em Transformation groups}, 9:337--379, 2004.

\bibitem[LMMR13]{Lotscher}
Roland L\"otscher, Mark MacDonald, Aurel Meyer, and Zinovy Reichstein.
\newblock Essential dimension of algebraic tori.
\newblock {\em J. Reine Angew. Math.}, 677:1--13, 2013.

\bibitem[Lor05]{Lorenz}
Martin Lorenz.
\newblock {\em Multiplicative invariant theory}, volume 135.
\newblock Springer Science \& Business Media, 2005.

\bibitem[LS74]{Landazuri}
Vicente Landazuri and Gary~M. Seitz.
\newblock On the minimal degrees of projective representations of the finite
  {C}hevalley groups.
\newblock {\em J. Algebra}, 32:418--443, 1974.

\bibitem[Mac11]{MacDonald}
Mark~L. MacDonald.
\newblock Essential {$p$}-dimension of the normalizer of a maximal torus.
\newblock {\em Transform. Groups}, 16(4):1143--1171, 2011.

\bibitem[Mer17]{Merkurjev}
Alexander Merkurjev.
\newblock Essential dimension.
\newblock {\em Bulletin of the American Mathematical Society}, 54(4):635--661,
  2017.

\bibitem[Mor21]{Moreto}
Alexander Moret{\'o}.
\newblock On the minimal dimension of a faithful linear representation of a
  finite group.
\newblock {\em arXiv preprint arXiv:2102.01463}, 2021.

\bibitem[MR10]{Meyer}
Aurel Meyer and Zinovy Reichstein.
\newblock Some consequences of the {K}arpenko-{M}erkurjev theorem.
\newblock {\em Doc. Math.}, pages 445--457, 2010.

\bibitem[Nuz95]{Nuzhin}
Ya.~N. Nuzhin.
\newblock Weyl groups as {G}alois groups of a regular extension of the field
  {$\bold Q$}.
\newblock {\em Algebra i Logika}, 34(3):311--315, 364, 1995.

\bibitem[Ple85]{Plesken}
Wilhelm Plesken.
\newblock Finite unimodular groups of prime degree and circulants.
\newblock {\em Journal of Algebra}, 97(1):286--312, 1985.

\bibitem[PP77a]{PleskenPohst1}
Wilhelm Plesken and Michael Pohst.
\newblock On maximal finite irreducible subgroups of $\operatorname{GL}\left(n,
  \mathbb{Z}\right)$. {I}. {T}he five and seven dimensional cases.
\newblock {\em Mathematics of Computation}, 31(138):536--551, 1977.

\bibitem[PP77b]{PleskenPohst2}
Wilhelm Plesken and Michael Pohst.
\newblock On maximal finite irreducible subgroups of $\operatorname{GL}\left(n,
  \mathbb{Z}\right)$. {I}{I}. {T}he six dimensional case.
\newblock {\em Mathematics of Computation}, 31(138):552--573, 1977.

\bibitem[PP80a]{PleskenPohst3}
Wilhelm Plesken and Michael Pohst.
\newblock On maximal finite irreducible subgroups of $\operatorname{GL}\left(n,
  \mathbb{Z}\right)$. {I}{I}{I}. {T}he nine dimensional case.
\newblock {\em Mathematics of Computation}, 34(149):245--258, 1980.

\bibitem[PP80b]{PleskenPohst4}
Wilhelm Plesken and Michael Pohst.
\newblock On maximal finite irreducible subgroups of $\operatorname{GL}\left(n,
  \mathbb{Z}\right)$. {I}{V}. {R}emarks on even dimensions with applications to
  $n = 8$.
\newblock {\em Mathematics of Computation}, 34(149):245--258, 1980.

\bibitem[PP80c]{PleskenPohst5}
Wilhelm Plesken and Michael Pohst.
\newblock On maximal finite irreducible subgroups of $\operatorname{GL}\left(n,
  \mathbb{Z}\right)$. {V}. {T}he eight dimensional case and a complete
  description of dimensions less than ten.
\newblock {\em Mathematics of Computation}, 34(149):277--301, 1980.

\bibitem[Rei70]{Reiner}
Irving Reiner.
\newblock A survey of integral representation theory.
\newblock {\em Bulletin of the American Mathematical Society}, 76(2):159--227,
  1970.

\bibitem[Ry{\v{s}}72]{Ryskov}
SS~Ry{\v{s}}kov.
\newblock Maximal finite groups of integral nxn matrices and full groups of
  integral automorphims of positive quadratic forms ({B}ravais models).
\newblock In {\em Trudy Mat. Inst. Steklov, Proc. Steklov Inst. Math}, volume
  128, pages 217--250, 1972.

\bibitem[Sou94]{Souvignier}
Bernd Souvignier.
\newblock Irreducible finite integral matrix groups of degree 8 and 10.
\newblock {\em Mathematics of Computation}, 63(207):335--350, 1994.

\bibitem[Suz60]{Suzuki}
Michio Suzuki.
\newblock A new type of simple groups of finite order.
\newblock {\em Proc. Nat. Acad. Sci. U.S.A.}, 46:868--870, 1960.

\bibitem[SZ93]{Seitz}
Gary~M. Seitz and Alexander~E. Zalesskii.
\newblock On the minimal degrees of projective representations of the finite
  {C}hevalley groups. {II}.
\newblock {\em J. Algebra}, 158(1):233--243, 1993.

\bibitem[Tah71]{Tahara}
Ken-Ichi Tahara.
\newblock On the finite subgroups of $\operatorname{GL}\left(3,
  \mathbb{Z}\right)$.
\newblock {\em Nagoya Mathematical Journal}, 41:169--209, 1971.

\bibitem[Vos65]{Voskresenskii}
Valentin~Evgen'evich Voskresenskii.
\newblock On two-dimensional algebraic tori.
\newblock {\em Izvestiya Rossiiskoi Akademii Nauk. Seriya Matematicheskaya},
  29(1):239--244, 1965.

\bibitem[War66]{Ward}
Harold~N. Ward.
\newblock On {R}ee's series of simple groups.
\newblock {\em Trans. Amer. Math. Soc.}, 121:62--89, 1966.

\bibitem[Wat79]{Waterhouse}
William~C. Waterhouse.
\newblock {\em Introduction to affine group schemes}, volume~66 of {\em
  Graduate Texts in Mathematics}.
\newblock Springer-Verlag, New York-Berlin, 1979.

\bibitem[WBN71]{Wondratschek}
H~Wondratschek, R~B{\"u}low, and J~Neub{\"u}ser.
\newblock On crystallography in higher dimensions. {I}, {I}{I}, {I}{I}{I}.
\newblock {\em Acta Crystallographica Section A: Crystal Physics, Diffraction,
  Theoretical and General Crystallography}, 27(6):517--535, 1971.

\bibitem[Wil09]{Wilson}
Robert~A. Wilson.
\newblock {\em The finite simple groups}, volume 251 of {\em Graduate Texts in
  Mathematics}.
\newblock Springer-Verlag London, Ltd., London, 2009.

\end{thebibliography}

\end{document}